\documentclass[12pt,a4paper]{article}
\usepackage[top=2.5cm,bottom=2.5cm,left=2.2cm,right=2.2cm]{geometry}
\def\blue{\color{blue}}

\usepackage[T1]{fontenc}
\usepackage[utf8]{inputenc}
\usepackage{color}
\usepackage{graphicx}
\usepackage{amsfonts}
\usepackage{extarrows}
\usepackage{amsmath,amsthm,amssymb,color}
\usepackage{hyperref}
\usepackage{eepic}
\usepackage{lineno}
\usepackage{enumerate}	
\usepackage{paralist}
\usepackage{algorithm}
\usepackage{algorithmicx}
\usepackage{algpseudocode}
\usepackage{authblk}
\usepackage{mathtools}
\usepackage{pifont}
\usepackage[numbers,sort&compress]{natbib}

\usepackage{tikz}
\hypersetup
{
	colorlinks=true,
	linkcolor=blue,
	filecolor=blue,
	urlcolor=blue,
	citecolor=cyan,
}
\def\blue{\color{blue}}

\usepackage[nameinlink]{cleveref}
\newtheorem{lemma}{Lemma}[section]
\crefname{lemma}{Lemma}{Lemmas}
\newtheorem{theorem}[lemma]{Theorem}
\crefname{theorem}{Theorem}{Theorems}

\newtheorem{claim}[]{Claim}
\crefname{claim}{Claim}{Claims}

\theoremstyle{definition}
\newtheorem{remark}[lemma]{Remark}

\newtheorem{notation}[lemma]{Notation}
\newtheorem{case}{\indent Case}[section]

\usepackage{authblk}

\newenvironment{wst}
{\setlength{\leftmargini}{1.5\parindent}
	\begin{itemize}
		\setlength{\itemsep}{-1.1mm}}
	{\end{itemize}}
\begin{document}
	\title{The circumference of a graph with given minimum degree and clique number}
	\author{Na Chen\footnote{Email: na$\_$chen@126.com.}}	
	\author{~Yurui Tang\footnote{Email: tyr2290@163.com.}}
	\affil{\small Department of Mathematics,
		East China Normal University, Shanghai, 200241, China}
	\date{}
	\maketitle
	
	\begin{abstract}
		The circumference denoted by $c(G)$ of a graph $G$ is the length of its longest cycle. Let $\delta(G)$ and $\omega(G)$ denote the minimum degree and the clique number of a graph $G$, respectively. In [\emph{Electron. J. Combin.} 31(4)(2024) $\#$P4.65], Yuan proved that if $G$ is a 2-connected graph of order $n$, then $c(G)\geq \min\{n,\omega(G)+\delta(G)\}$ unless $G$ is one of two specific graphs.
		In this paper, we prove a stability result for the theorem of Erd\H os and Gallai, thereby helping us to characterize all $2$-connected non-hamiltonian graphs whose circumference equals the sum of their clique number and minimum degree. Combining this with Yuan's result, one can deduce that if $G$ is a $2$-connected graph of order $n$, then $c(G)\geq \min\{n,\omega(G)+\delta(G)+1\}$, unless $G$ belongs to certain specified graph classes.
		
	\end{abstract}
	
	
	
	\section{Introduction}
	
	We consider finite simple graphs. For any undefined terminology or notation, we refer the reader to the books \cite{BM, W}. The \emph{order} of a graph is its number of vertices, and the \emph{circumference} denoted by $c(G)$ of a graph $G$ is the length of its longest cycle. The \emph{clique number} of a graph $G$, written $\omega(G)$, is the maximum cardinality of a clique in $G$. Let $\delta(G)$ denote the minimum degree of a graph $G$.
	
	A $k$-cycle is a cycle of length $k$. We denote by $K_n$ and $K_{s,t}$ the complete graph of order $n$ and the complete bipartite graph whose partite sets have cardinality $s$ and $t$, respectively. Let $\overline{G}$ denote the complement of a graph $G$.
	For graphs we will use equality up to isomorphism, so $G = H$ means that $G$ and $H$ are isomorphic.  For two graphs $G$ and $H$, $H\subseteq G$ means that $H$ is a subgraph of $G$ and $G\vee H$ denotes the join of $G$ and $H$, which is obtained from the disjoint union $G+H$ by adding edges joining every vertex of $G$ to every vertex of $H.$
	
	The study of the longest cycle in graphs has been a central topic in extremal and structural graph theory. In 1952, a fundamental result of Dirac \cite{D} states that any 2-connected graph $G$ of order $n$ satisfies  $c(G)\ge \min\{n,2\delta(G)\}$.
	Ore \cite{O} proved that for any 2-connected non-hamiltonian graph $G$ with $\delta(G)=k$, $c(G)=2k$ holds precisely when $\overline{K}_k\vee \overline{K}_s\subseteq G\subseteq K_k\vee \overline{K}_s$ for some $s\ge k+1.$
	Voss \cite{V} strengthened Dirac's theorem by raising the lower bound on $c(G)$ by two, and provided a characterization of all extremal graphs achieving this bound.
	In 2023, a new proof of Voss's theorem was presented by Zhu, Gy\H ori, He, Lv, Salia, and Xiao \cite{Z}, where they further highlighted the theorem's versatility by applying it to a range of generalized Tur$\acute{\text{a}}$n problems. Bondy \cite{B} proved that for any 2-connected graph $G$ of order $n$, if every vertex except for at most one vertex is of degree at least $k$, then $c(G)\ge \min \{2k,n\}.$
	In 2025, Ning and Yuan \cite{BY} proved a stability result of Bondy's theorem which also improved Voss's theorem.
	For other relevant work, see \cite{Li,LN,MN,S}.

	For integers $n\geq \omega+\delta>2\delta$, let $H(n,\omega,\delta)=K_\delta\vee (K_{\omega-\delta}+\overline{K}_{n-\omega}).$
	For integers $n=\omega-2+l(\delta-1)+2$ with $\omega>\delta$ and $l\geq 2$, let $Z(n,\omega,\delta)=K_2\vee (K_{\omega-2}+lK_{\delta-1})$.
	Clearly, the minimum degree, the clique number and the circumference of $H(n,\omega,\delta)$ and $Z(n,\omega,\delta)$ are $\delta$, $\omega$ and $\omega+\delta-1$, respectively.
	
	In 2024, Yuan \cite{Yuan} proved the following result on the circumference involving the minimum degree and the clique number of a graph.
	
	\begin{theorem}[Yuan \cite{Yuan}]\label{Yuan}
		Let $G$ be a $2$-connected graph of order $n$ with clique number $\omega$ and minimum degree $\delta$. Then $c(G)\ge \min\{n,\omega+\delta\}$ unless $G=H(n,\omega,\delta)$ or $Z(n,\omega,\delta)$.
	\end{theorem}
	
	Moreover, Yuan \cite{Y} used {\blue Theorem \ref{Yuan}} to prove a longstanding conjecture of Erd\H os, Simonovits and S$\acute{\text{o}}$s \cite{E}(determining the maximum number of edge colors in a complete graph such that there is no rainbow path of given length).
	In 2024, Ma and Yuan \cite{M} improved the result of  F$\ddot{\text{u}}$redi, Kostochka, Luo and Verstra$\ddot{\text{e}}$te in \cite{F,FL} by combining {\blue Theorem \ref{Yuan}} with a stability result of the well-known P$\mathrm{\acute{o}}$sa lemma.
	
	This paper aims to establish a stability version of {\blue Theorem \ref{Yuan}} by characterizing all 2-connected non-hamiltonian graphs whose circumference equals the sum of their clique number and minimum degree. Before presenting our main result, we first introduce the following graph classes.
	
	\begin{notation}
		For integers $n\geq \omega+\delta+1\ge2\delta+1$, let $H_1(n,\omega+1,\delta)=K_\delta\vee (K_{\omega+1-\delta}+\overline{K}_{n-\omega-1})$. For integers $n\geq \omega+\delta+1\ge2\delta+3$, let $H_1(n,\omega,\delta+1)=K_{\delta+1}\vee (K_{\omega-1-\delta}+\overline{K}_{n-\omega})$.
		For integers $n=\omega-1+l(\delta-1)+2$ with $\omega\ge\delta$ and $l\geq 2$, let $H_2(n,\omega+1,\delta)=K_2\vee (K_{\omega-1}+lK_{\delta-1})$ (see Figure \ref{fig1}).
	\end{notation}
	
	Clearly, the minimum degree, the clique number and the circumference of $H_1(n,\omega+1,\delta)$ and $H_2(n,\omega+1,\delta)$ are $\delta$, $\omega+1$ and $\omega+\delta$, respectively. And the minimum degree, the clique number and the circumference of $H_1(n,\omega,\delta+1)$ are $\delta+1$, $\omega$ and $\omega+\delta$, respectively.
	
	\begin{figure}[h]
		\centering
		\includegraphics[width=0.7\textwidth]{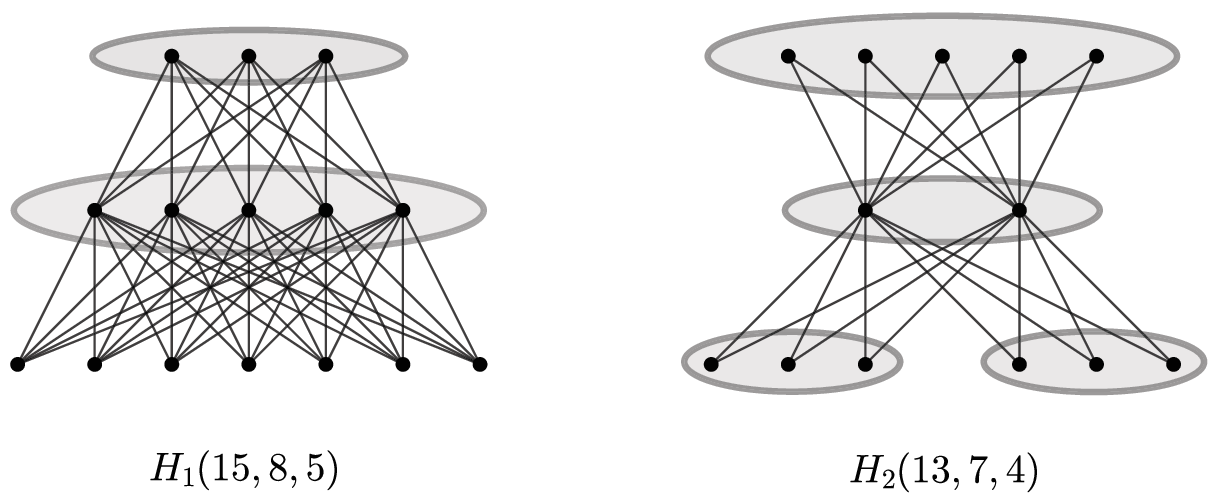}
		\caption{\footnotesize{The vertices in each gray ellipse induce a complete graph.}}
		\label{fig1}
	\end{figure}

	\begin{notation}
		For integers $n=\omega+b_1+b_2$ with $\omega\ge 4$ and $\min\{b_1,b_2\}\ge 2,$ let $H_3(n,\omega,2)$ be the graph of order $n$ with a vertex partition $A\cup B_1\cup B_2$ with sizes $\omega$, $b_1$ and $b_2$ respectively, whose edge set consists of all edges in $A$, all edges between $\{v_1,v_2\}\subseteq A$ and $B_1$ and all edges between $\{v_3,v_4\}\subseteq A$ and $B_2,$ where $\{v_1,v_2\}\cap \{v_3,v_4\}=\emptyset$ (see Figure \ref{fig2}).
	\end{notation}
	Note that $\delta(H_3(n,\omega,2))=2$, $\omega(H_3(n,\omega,2))=\omega$ and $c(H_3(n,\omega,2))=\omega+2$.
	
	\begin{notation}
		For integers $n=\omega+l_1+2l_2$ with $\omega> 3$ and $l_1+l_2\ge 3$,
		let $R_1=C_1\vee(K_{\omega-3}+\overline{K}_{l_1})$ with $C_1=K_3$. Let $H_4(n,\omega,3)$ be the graph obtained from $R_1$ by adding $l_2K_2$, each joining the same two common vertices of $C_1$  (see Figure \ref{fig2}).
	\end{notation}
	Note that $\delta(H_4(n,\omega,3))=3$, $\omega(H_4(n,\omega,3))=\omega$ and $c(H_4(n,\omega,3))=\omega+3$.
	
	\begin{figure}[h]
		\centering
		\includegraphics[width=0.7\textwidth]{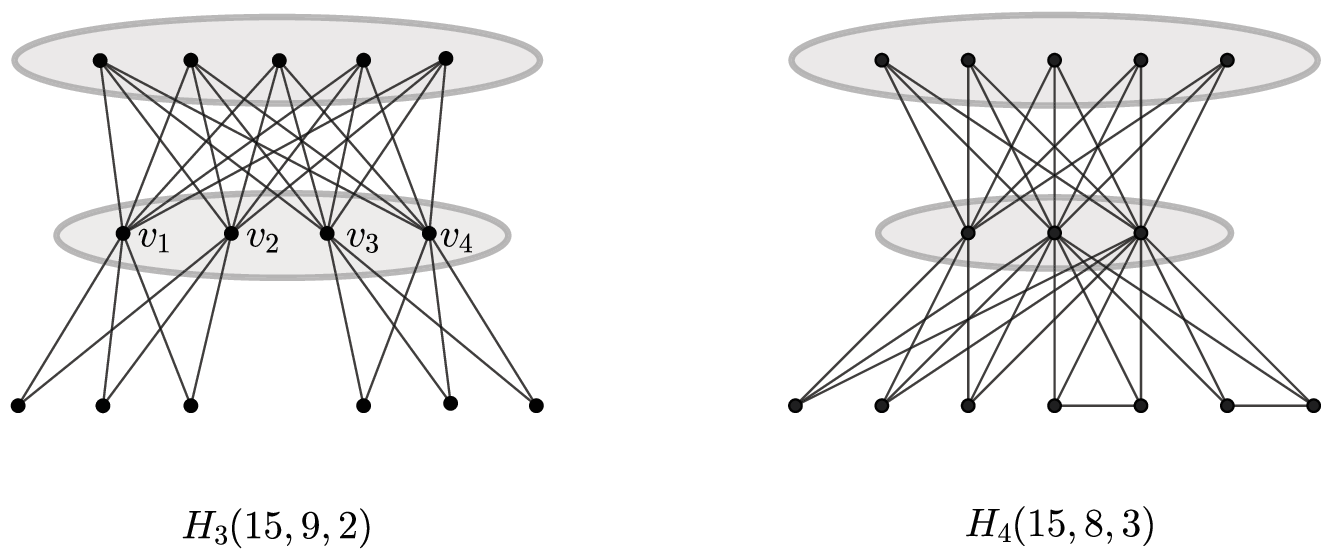}
		\caption{\footnotesize{The vertices in each gray ellipse induce a complete graph.}}
		\label{fig2}
	\end{figure}
	
	\begin{notation}
		For integers $n\ge \omega+\delta+1>2\delta+1$, let $G_1=K_\delta\vee (K_{\omega-\delta}+K_2+\overline{K}_{n-\omega-2})$(see Figure \ref{fig3}).
	\end{notation}
	Note that $\delta(G_1)=\delta$, $\omega(G_1)=\omega$ and $c(G_1)=\omega+\delta$.
	
	For positive integers $n_i,h_i$, $\delta$ and $l$ with  $n_i=h_i(\delta-1)+1$ and $h_i\ge 2$, let $L_i(\delta)=K_1\vee h_iK_{\delta-1}$.
	Let $L(\delta)=L_1(\delta)+L_2(\delta)+\dots+L_l(\delta)$, and let $S$ is the union of stars of order at least four.
	
	\begin{notation}
		For integers $n=\omega+(l_1+l_3)(\delta-1)+l_2\delta+|S|+|L(\delta)|$ with $l_1+l_2\ge 1$, $|S|\ge 0$, $|L(\delta) |\ge 0$, $l_3\ge 2$ and $\omega > \delta,$ let $T_1=A_1\vee(K_{\omega-2}+l_1K_{\delta-1}+l_2K_\delta+S)$ where $A_1=a_1a_2$ is an edge and let $T_2=A_2\vee l_3 K_{\delta-1}$ with $A_2=K_2$.  
		Let $T_3$ be the graph of order $n$ from $T_1$ and $L(\delta)$ by taking the join of $a_1$ with $L(\delta)$ and by adding edges from $a_2$ to the cut-vertex of each component of $L(\delta)$.
		Let $G_2=T_3$, where $T_3$ satisfies  two conditions: $T_3$ is non-hamiltonian and $c(T_3)=\omega+\delta$ (see Figure \ref{fig3}).
		Let $G_3$ be the graph obtained from $T_2$ and $T_3$ by identifying a vertex from $A_1$ and a vertex from $A_2$ as a new vertex $v$ and adding an edge between $A_1-v$ and $A_2-v$ (see Figure \ref{fig3}).
	\end{notation}
	
	For $i\in \{2,3\}$,  $\omega(G_i)=\omega$, $c(G_i)=\omega+\delta$.
	If $S\ne \emptyset$, then $\delta(G_i)=3$; 
	if $S=\emptyset$, then $\delta(G_i)=\delta$.  
	
	\begin{notation}
		For integers $n\geq \omega+3$ and $\omega\ge 4$, let $F_1$ be the graph with its vertex set partitioned into sets $A$, $B$ and $C$ of sizes $\omega-2$, $2$ and $3$, respectively, such that $F_1[A\cup B]=K_{\omega}$, $F_1[C]=\overline{K}_{3}$ and $[B,C]$ is complete. Let $a_1,a_2\in A$, let $B=\{b_1,b_2\}$ and let
		$C=\{c_1,c_2,c_3\}$. Let $F_2$ be the graph obtained from $F_1$ by adding two edges $a_1c_1$, $a_2c_2$ and deleting two edges $b_2c_1,b_2c_2$. 
		Let $G_4$ be the graph of order $n$ obtained by adding $n-\omega-3$ new vertices to $F_2$, where the neighborhood of each new vertex is exactly equal to that of one of $c_1$, $c_2$, or $c_3$ (see Figure \ref{fig3}).
	\end{notation}
	
	Note that $\delta(G_4)=2$, $\omega(G_4)=\omega$ and $c(G_4)=\omega+2$.
	
	\begin{figure}[h]
		\centering
		\includegraphics[width=0.65\textwidth]{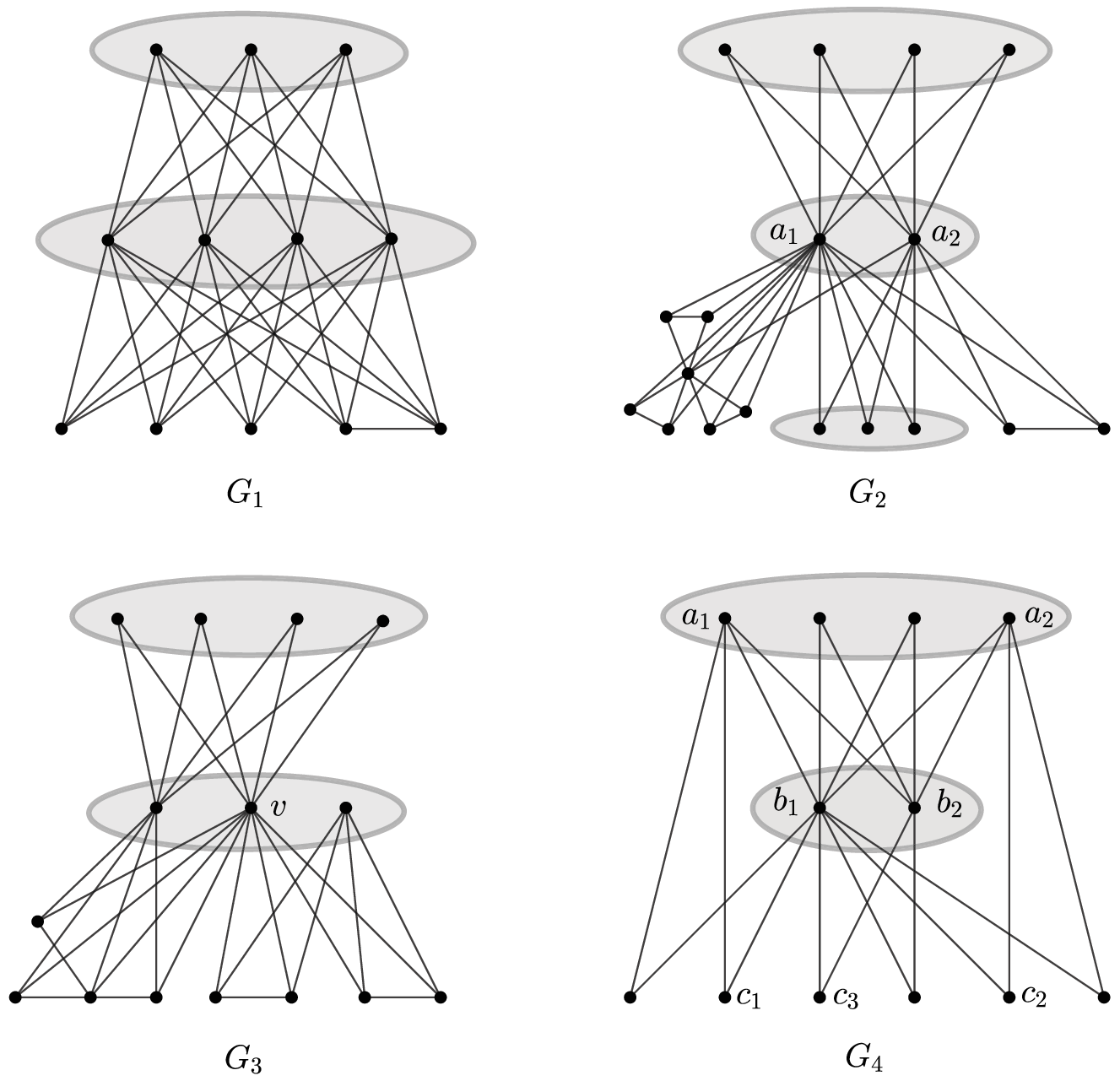}
		\caption{\footnotesize{The vertices in each gray ellipse induce a complete graph.}}
		\label{fig3}
	\end{figure}
	
	Let $\mathcal F=\{H_1(n,\omega+1,\delta),H_1(n,\omega,\delta+1),H_2(n,\omega+1,\delta), G_1,G_2,G_3,G_4\}.$
	Let $\mathcal G$ be the family of subgraphs of graphs in $\mathcal F$ that have order $n$, minimum degree $\delta$, and clique number $\omega$,  except for $H(n,\omega,\delta)$ and $Z(n,\omega,\delta)$.
	Graphs satisfying the above conditions can be easily found among the subgraphs of graphs in $\mathcal F$.
	Now, we establish the following main result.
	
	\begin{theorem}\label{thm1.3}
		Let $G$ be a $2$-connected graph of order $n$ with clique number $\omega$ and minimum degree $\delta$. Then $c(G)\ge \min\{n,\omega+\delta+1\}$ unless
		\begin{wst}
			\item[{\rm (i)}]
			$c(G)=\omega+\delta-1$ and $G\in \{H(n,\omega,\delta), Z(n,\omega,\delta)\};$ or
			
			\item[{\rm (ii)}]
			$c(G)=\omega+\delta$ and $G\in \{H_3(n,\omega,2),H_4(n,\omega,3)\}\cup \mathcal{G}.$
		\end{wst}
	\end{theorem}
	The remainder of this paper is organized as follows.
	In Section \ref{sec2}, we give some necessary notations and definitions.
	In Section \ref{sec3}, we present some lemmas for the proof.
	Section \ref{sec4} is devoted to the proof of {\blue Theorem \ref{thm1.3}}.

	\section{Notations and definitions}\label{sec2}
	
	We denote  by $V(G)$ the vertex set of a graph $G$, by $E(G)$ the edge set of $G$, and denote by $|G|$ the order of $G$.
	Let $N_G(v)$ be the neighborhood in $G$ of a vertex $v$, and let $d_G(x)$ be the size of $N_G(x)$,  and define $N_G[x]=N_G(x)\cup \{x\}$.
	For $S\subseteq V(G)$, let $N_G(S)=(\mathop{\cup}\limits_{x\in S}N_G(x))\setminus S$ and let $N_G[S]=N_G(S)\cup S$. Denote by $G[S]$ the subgraph induced by $S$.
	Let $A$ and $B$ be any two subsets of $V(G)$. We write $[A,B]$  to denote the set of edges that have one end in $A$ and the other end in $B$. If $G[[A,B]]$  is a complete bipartite graph, then we say $[A,B]$  is complete.
	For any pair binary tuple $(x,y)$ and $(a,b)$, we denote $x=a$ and $y=b$ by $(x,y)=(a,b)$. 
	
	For two distinct vertices $x$ and $y$, an \emph{$(x,y)$-path} is a path whose endpoints are $x$ and $y$.
	We may extend the notion of an $(x,y)$-path to paths connecting subsets $X$ and $Y$ of $V(G)$.
	An \emph{$(X,Y)$-path} is a path which starts at a vertex of $X$, ends at a vertex of $Y$, and whose internal vertices belong to neither $X$ nor $Y$; if $F_1$ and $F_2$ are subgraphs of a graph $G$, we write \emph{$(F_1,F_2)$-path} instead of $(V(F_1),V(F_2))$-path.
	
	Recall that every connected graph $G$ has a block-cutvertex tree (\cite{BM}, \cite{W}), a leaf of which is called an \emph{end-block}. If $G$ has cut-vertices, then an end-block of $G$ contains exactly one cut-vertex. If $B$ is an end-block and a vertex $b$ is the only cut-vertex of $G$ with $b \in V(B)$, then we say that $B$ is an\emph{ end-block with cut-vertex} $b$.
	
	Let $P=v_1v_2\cdots v_m$ be a path in a graph $G$.
	For $v_i,v_j\in V(P)$, we use $v_iPv_j$ to denote the subpath of $P$ between $v_i$ and $v_j$.
	For a vertex $v\in V(P)$, denote $v^-$ and $v^+$ to be the immediate predecessor and successor of $v$ on $P$, respectively.
	Let $v^{-1}=v^-$ and $v^{+1}=v^+$.
	Let $v^{-i}=(v^{-(i-1)})^-$ and $v^{+i}=(v^{+(i-1)})^+$ for $i\geq 2$.
	For $S\subseteq V(P)$, let $S^+=\{v^+:v\in S\setminus\{v_m\}\}$ and $S^-=\{v^-:v\in S\setminus\{v_1\}\}$.
	We call $(i,j)$ a \emph{crossing pair} of $P$ if $v_i\in N_P(v_m)$ and $v_j\in N_P(v_1)$ with $i<j$.
	A crossing pair $(i,j)$ is \emph{minimal} in $P$ if $v_h\notin N_P(v_1)\cup N_P(v_m)$ for each $i<h<j$. For any minimal crossing pair $(i,j)$, we call $(i,j)$ is a \emph{minimum} crossing pair if $j-i$ is minimum.
	We say a vertex $x\notin Y$ is  \emph{connected to} a vertex set $X\subseteq Y$ if there is a path starting from $x,$ ending at $x'\in X$ and without containing any vertices of $Y\setminus \{x'\}$.

	\section{Preliminaries and Lemmas}\label{sec3}
	
	We need the following well-known lemma proved by P$\acute{\text{o}}$sa \cite{Po}.
	
	\begin{lemma}[P$\mathrm{\acute{o}}$sa \cite{Po}]\label{L1}
		Let $G$ be a $2$-connected graph of order $n$ with a path $P=x_1x_2\dots x_k$. Then $c(G)\ge \min\{n,d_P(x_1)+d_P(x_k)\}.$
		Furthermore,  if $P$ does not contain a crossing pair and $N_P(x_1)\cap N_P(x_k)\neq\emptyset$, then $c(G)\ge \min\{n,d_P(x_1)+d_P(x_k)+1\}$.
		If $P$ does not contain a crossing pair and $N_P(x_1)\cap N_P(x_k)=\emptyset$, then $c(G)\ge \min\{n,d_P(x_1)+d_P(x_k)+2\}$.
	\end{lemma}
	
	Erd\H os and Gallai \cite{EG} proved the following result which has been used in several papers.
	\begin{lemma}[Erd\H os and Gallai \cite{EG}]\label{EG}
		Let $G$ be a 2-connected graph and $x,y$ be two given vertices. If every vertex other than $x,y$ has a degree at least $k$ in $G$, then there is an $(x,y)$-path of length at least $k$.
	\end{lemma}

	For two distinct vertices $x,y\in V(G)$, we call the triple $(G,x,y)$ is a \emph{rooted graph} with roots $x$ and $y$. The \emph{minimum degree} of  $(G,x,y)$ is defined to be $ \min\{d_G(v)~|~ v\in V(G)\setminus\{x,y\}\}.$  $(G,x,y)$ is said to be \emph{$2$-connected} if $G+xy$ is $2$-connected, where $G+xy=G$ if $xy\in E(G)$ and $G+xy$ is the graph obtained from $G$ by adding the edge $xy$ if $xy\notin E(G)$.
	
	Let $T(\delta)$ be the union of components of order $\delta$ or $\delta-1$. A useful generalization of \Cref{EG} is the following lemma, which helps characterize the local structure of the extremal graphs.
	
	\begin{lemma}\label{L2.2}
		For an integer $\delta\ge 2$,
		let $(G,x,y)$ be a $2$-connected rooted graph with minimum degree at least $\delta$ and order at least $\delta+3$. Then there exists an $(x,y)$-path of order at least $ \delta+3$ in $(G,x,y)$, unless
		$G-\{x,y\}=L(\delta)+ T(\delta)$, or
		$\delta=3$ and $G-\{x,y\}=S+ L(3)+T(3)$. 
		In particular, 
		for each component of $L(\delta)$, if $x$ is adjacent to a non-cutvertex in that component, then $y$ is adjacent only to the cut-vertex of the same component, and vice versa.
	\end{lemma}
	
	\begin{proof}
		Note that adding the edge $xy$ does not affect  the order of a longest $(x,y)$-path in $G$. So we may assume that $x$ and $y$ are adjacent in $(G,x,y)$. Base on the definition of a rooted graph, $G$ is $2$-connceted.  
		Let $H=G-\{x,y\}$ and then $|H|\ge \delta+1$. 
		
		Suppose $H$ is connected.
		We use induction on $\delta$ to prove the statement.
		
		\textbf{The basic step for $\delta=2.$}
		If $H$ is $2$-connected, then $H$ contains a cycle $C$. Since $G$ is $2$-connected, there exist two vertex-disjoint $(\{x,y\},V(C))$-paths. Then we can find an $(x,y)$-path of order at least five in $G$. If $H$ is not $2$-connected, then $H$ has a cut-vertex. Consider its block-cutvertex tree. Since $G$ is $2$-connected, 
		there exists an end-block $B_1$ with cut-vertex $b_1$ such that one of the vertices $x$ and $y$ has a neighbor in $B_1-b_1$, and the other has a neighbor in $H-V(B_1-b_1)$.
		Without loss of generality, let $u\in N(x)\cap V(B_1-b_1)$ and $v\in N(y)\cap V(H-V(B_1-b_1)).$ 
		
		Suppose $v\in N(y)\cap V(H-V(B_1)).$
		Since $H$ contains a $(u,v)$-path $P$ of order at least three,  $xuPvy$ is an $(x,y)$-path of order at least five in $G$.
		Suppose $N(y)\cap V(H-V(B_1))=\emptyset$, that is, $v=b_1$. If $H$ contains another end-block $B_2$ with cut-vertex $b_2\ne b_1$, then $x$ has a neighbor in $B_2$ since $G$ is $2$-connected. Hence, we can find an $(x,y)$-path of order at least five in $G$.
		Suppose $H$ contains exactly one cut-vertex $b_1$. Now $N_H(y)=\{b_1\}$. If there is an end-block of order at least three, then we can easily find an $(x,y)$-path of order at least five. If each end-block has order two, then $H=L(2).$  
		
		\textbf{The induction step for $\delta\ge 3.$} Suppose \Cref{L2.2} holds for all $2$-connected rooted graphs with minimum degree less than $\delta$. 
		Suppose that $H$ is $2$-connected.
		Let $H'=G-x$.
		Since $G$ is $2$-connected, there is a vertex $u\in N(x)\cap V(H)$. Now $(H',u,y)$ is a $2$-connected rooted graph of order at least $\delta+2$ with minimum degree at least $\delta-1\ge 2.$ Since $H=H'-y$ is $2$-connected, $H'-\{u,y\}$ is connected. Note that $|H|\ge \delta+1$.
		
		By the induction hypothesis, there exists a $(u,y)$-path $P$ of order at least $ \delta+2$ in $H'$ and hence $xuPy$ is an $(x,y)$-path of order at least $\delta+3$ in $G$, unless $H'-\{u,y\}=L(\delta-1)$ or $\delta=4$ and $H'-\{u,y\}=S$.
		Since $H'-\{u,y\}$ is connected, both $L(\delta-1)$ and $S$ have exactly one component.
		Suppose $H'-\{u,y\}=L(\delta-1)$. Let $v_l$ be the cut-vertex of $L(\delta-1)$. Note that $u$ or $y$ is only adjacent to $v_l$ in $L(\delta-1)$.
		Without loss of generality, let $N_{L(\delta-1)}(y)=\{v_l\}$. By $\delta((G,x,y))\ge \delta,$  $[V(L(\delta-1)-v_l),\{x,u\}]$ is complete. We can find an $(x,y)$-path of order at least $\delta+3$. 
		Suppose $H'-\{u,y\}=S=K_{1,t}$ with $t\ge 3$. Let $v_t$ be the cut-vertex of $K_{1,t}$.
		By $\delta((G,x,y))\ge \delta=4,$ $[V(S-v_t),\{x,y,u\}]$ is complete. We can  find an $(x,y)$-path of order seven.
		Thus, there always exists an $(x,y)$-path of order at least $ \delta+3$ in $G$.
		
		Assume that $H$ has a cut-vertex and consider its block-cutvertex tree. If there exists an $(x,y)$-path of order $\delta+3$ in $G,$ we have done, so assume that $G$ contains no $(x,y)$-path of order at least $\delta+3.$
		For $1\le i \le s,$ let $B_i$ be an end-block with cut-vertex $b_i$.  Since $G$ is $2$-connected, without loss of generality, suppose $N(x)\cap V(B_j-b_j)\ne \emptyset$ for some $1\le j\le s$ and let $H_j=G[V(B_j)\cup\{x\}]$. 
		
		We now distinguish two cases based on the locations of the neighbors of $y$ in $H$.
		\vskip 3mm
		\textbf{Case~1.}{~$N(y)\cap V(B_i-b_i)= \emptyset$ for any $1\le i\le s$.}
		\vskip 3mm
		Since $G$ is $2$-connected, we assert that $N(x)\cap V(B_i-b_i)\ne \emptyset$ for any $1\le i\le s$. Otherwise, $G$ would have a cut-vertex, a contradiction. By $\delta((G,x,y))\ge \delta$ and $N(y)\cap V(B_i-b_i)= \emptyset$, we have $|B_i|\ge \delta$, for any $1\le i\le s$.
		\vskip 3mm
		{\bf Claim 1.} $H_j$ contains an $(x,b_j)$-path $P_j$ of order at least $\delta+1$ for any $1\le j\le s$.
		\begin{proof}
			Note that $(H_j,x,b_j)$ is a $2$-connected rooted graph of order at least $\delta+1$ with minimum degree at least $\delta$.
			If $\delta\ge 4$, by the induction hypothesis, there exits an $(x,b_j)$-path $P_j$ of order at least $\delta+1$ in $H_j$, unless $B_j-b_j=L(\delta-2)$, or $\delta=5$ and $B_j-b_j=S$.
			Since $B_j-b_j$ is connected,  both $L(\delta-2)$ and $S$ have exactly one component.
			Suppose $B_j-b_j=L(\delta-2)$. Let $v_l$ be the cut-vertex of $L(\delta-2)$. Note that $x$ or $b_j$ is only adjacent to $v_l$ in $L(\delta-2).$ Without loss of generality, let $N_{L(\delta-2)}(x)=\{v_l\}$. Then each vertex of $L_(\delta-2)\setminus\{v_l\}$ has degree at most $\delta-2$, a contradiction to $\delta((H_j,x,b_j))\ge \delta$. Suppose $B_j-b_j=S=K_{1,t_1}$ with $t_1\ge 3$, a contradiction to $\delta((H_j,x,b_j))\ge \delta=5.$
			Thus, there exists an $(x,y)$-path of order at least $\delta+1$.

			Suppose $\delta=3$. Recall that $(H_j,x,b_j)$ is a $2$-connected rooted graph of order at least four with minimum degree at least three. Let $u\in N(x)\cap V(B_j-b_j).$
			By $d_{H_j}(u)\ge 3$, there is a vertex $w\in N(u_1)\cap V(B_j-b_j)$.
			Since $G$ is $2$-connected, there are two internal vertex-disjoint $(w,\{u,b_j\})$-paths. Hence, there is an an $(x,b_j)$-path of order at least four in $H_j$.
			Thus, $H_j$ contains an $(x,b_j)$-path of order at least $\delta+1$. This proves Claim 1.
		\end{proof}
		
		Since $G$ is $2$-connected and $N(y)\cap V(B_i-b_i)= \emptyset$ for any $1\le i\le s$, we have  $[\{y\},V(H-\bigcup\limits_{i=1}^{s}V(B_i-b_i))]\ne\emptyset.$
		For $u,v\in V(G),$ denote $\phi(u,v)$ the length of a longest $(u,v)$-path. Let $\phi(u,A)=\max\limits_{v\in A}\phi(u,v).$ 
		We choose an end-block $B_k$ with $1\le k\le s$ such that $\phi(b_k,N_H(y))=\max\limits_{1\le i\le s}\phi(b_i,N_H(y))$.  Without loss of generality, let  $\phi(b_1,v)=\phi(b_k,N_H(y))$ and let $Q$ be a longest $(b_1,v)$-path. Clearly, $V(Q)\cap V(B_1)=\{b_1\}.$  We choose a longest $(x,b_1)$-path $P_1$. 
		By Claim 1, $|P_1|\ge \delta+1.$
		Now, there exists an $(x,y)$-path $P=xP_1b_1Qvy$ of order at least $\delta+2$. 
		
		By $|P|\le \delta+2$, we have $|P_1|=\delta+1$, $b_1=v$ and  $|P|=\delta+2.$
		By the maximality of $Q$, we have $N_H(y)\subseteq\{b_1,b_2,\cdots,b_s\}$. Hence, $H$ has exactly one cut-vertex. 
		If $|H_j|\ge \delta+2,$ then $(H_j,x,b_j)$ is a $2$-connected rooted graph of order at least $\delta+2$ with minimum degree at least $\delta$,
		similar to the proof of Claim 1, there exists an $(x,b_j)$-path of order $\delta+2$, a contradiction to the maximality of $P_1.$ Thus, $|H_j|=\delta+1$. By $\delta((G,x,y))\ge \delta$, $H_j$ is a complete graph of order $\delta+1$. Thus, $G-\{x,y\}=L(\delta)$, and $y$ is adjacent only to the cut-vertex of $L(\delta)$.
		
		\vskip 3mm
		\textbf{Case~2.} {$N(y)\cap V(B_i-b_i)\ne \emptyset$ for some $1\le i\le s$.}
		\vskip 3mm
		{\bf Claim 2.} For an end-block $B_j$, if $N(x)\cap V(B_j-b_j)\ne \emptyset$, then $H_j$ contains an $(x,b_j)$-path $P_j$ of order at least $\delta$.
		\begin{proof}
			Note that $(H_j,x,b_j)$ is a $2$-connected rooted graph of order at least $\delta$ with minimum degree at least $\delta-1$. We provide a brief proof of Claim 2, as it is similar to that of Claim 1.
			
			Suppose $\delta\ge 5$.
			Since $B_j$ is a block of order at least $ \delta-1$ and $d_{H_j}(u)\ge \delta-1$ for any $u\in V(B_j-b_j)$,
			by the induction hypothesis, there always exists an $(x,b_j)$-path of order at least $ \delta$ in $G$.
			Suppose $\delta=4$. Recall that $(H_j,x,b_j)$ is a $2$-connected rooted graph of order at least four with minimum degree at least three. Let $u_1\in N(x)\cap V(B_j-b_j)$. By $d_{H_j}(u_1)\ge 3$, there is a vertex $u_2\in N(u_1)\cap V(B_j-b_j)$. Since $B_j$ is $2$-connected, there exist two internal vertex-disjoint paths $u_1Q_1b_j$ and $u_2Q_2b_j$. Then $u_1u_2Q_2b_j$ is a path of order at least three and hence $xu_1u_2Q_2b_j$ is a path of order at least four in $H_j.$ 
			Suppose $\delta=3.$ Then there is an $(x,b_j)$-path of order at least three in $H_j$ since $B_j$ is a block of order at least two. 
			Thus, $H_j$ contains an $(x,b_j)$-path $P_j$ of order at least $\delta$. This proves Claim 2.
		\end{proof}

		Without loss of generality, suppose $N(x)\cap V(B_1-b_1)\ne \emptyset$ and
		$N(y)\cap V(B_2-b_2)\ne \emptyset$. 
		Note that $(G[V(B_2)\cup\{y\}], b_2,y)$ is a $2$-connceted rooted graph of order at least $\delta$ with minimum degree at least $\delta-1$.
		Similar to the proof of Claim 2, there exists a $(b_2,y)$-path of order at least $ \delta$ in $G[V(B_2)\cup\{y\}]$. Let $P_1$ be a longest $(x,b_1)$-path in $H_1$ and let $P_2$ be a longest $(b_2,y)$-path in $G[V(B_2)\cup\{y\}]$. Clearly, $|P_1|\ge \delta$ and $|P_2|\ge \delta.$
		There is a $(b_1,b_2)$-path $P_0$ such that $V(P_0)\cap (V(B_1)\cup V(B_2))=\{b_1,b_2\}$ in $H$.
		Then $xP_1b_1P_0b_2P_2y$ is an $(x,y)$-path $P$ of order at least $ 2\delta-2+|P_0|$ in $G.$ Recall that $|P|\le \delta+2$. By $\delta\ge3$, we have $\delta=3$, $|P_0|=1$ and $|P|=2\delta-2+|P_0|=5$. Now $b_1=b_2$, $|P_1|=\delta=3$ and $|P_2|=\delta=3$.
		It is easy to verify that both $B_1$ and $B_2$ are edges. 
		
		If $H$ has at least two cut-vertex, there is an end-block $B_f$ such that $b_f\ne b_1.$ Let $P_0'$ be a $(b_f,b_1)$-path in $H$. Note that either $N(x)\cap V(B_f-b_f)\ne \emptyset$ or $N(y)\cap V(B_f-b_f)\ne \emptyset$. If $N(x)\cap V(B_f-b_f)\ne \emptyset$, then there is an $(x,y)$-path $P'=xP_fb_fP_0'b_2P_2y$ of order at least $\delta+3.$ If $N(y)\cap V(B_f-b_f)\ne \emptyset$, then there is an $(x,y)$-path $P''=xP_1b_1P_0'b_fP_fy$ of order at least $\delta+3$. Both situations contradict  $|P|\le \delta+2.$ Hence,  $H$ has exactly one cut-vertex. Clearly, each block is an edge.
		By $|G|\ge \delta+3=6,$ $V(G)\setminus V(P)\ne \emptyset.$ There exist at least three end-blocks in $H$. Then $H=K_{1,t}$ where $t\ge 3.$
		
		Thus, if $H$ has a cut-vertex, then either there exists an $(x,y)$-path of order at least $ \delta+3$ in $G$, unless $H=L_1(\delta)$, or $\delta=3$ and $H=K_{1,t}$, where $L_1(\delta)$ is the unique component of $L(\delta)$ and $t\ge 3$.
		
		Now, we consider that $H$ is disconnected. Let $H=T_1\cup T_2\cup\cdots \cup T_k$ where $k\ge2$ and $T_i$ is a component of $H$ for $1\le i\le k$.
		By $\delta((G,x,y))\ge \delta$, each component of $H$ has order at least $\delta-1.$ If there exists no component of order at least $\delta+1,$ then we are done, so assume that $H$ has some components of order at least $\delta+1.$ 
		Select an arbitrary component $T_i$  of $H$ with $|T_i|\ge \delta+1$. 
		Similar to the above argument, we can find an $(x,y)$-path of order at least $ \delta+3$ in $G[V(T_i)\cup\{x,y\}]$, unless $T_i=L_i(\delta)$, or $\delta=3$ and $T_i=K_{1,t_i}$, where $L_i(\delta)$ is the unique component of $L(\delta)$ and $t_i\ge 3$.

		Thus, there exists an $(x,y)$-path of order at least $ \delta+3$ in $(G,x,y)$, unless
		$G-\{x,y\}=L(\delta)+ T(\delta)$, or
		$\delta=3$ and $G-\{x,y\}=S+ L(3)+T(3)$ (see Figure \ref{fig4}). 
		In particular, 
		for each component of $L(\delta)$, if $x$ is adjacent to a non-cutvertex in that component, then $y$ is adjacent only to the cut-vertex of the same component, and vice versa.
		This proves \Cref{L2.2}.
	\end{proof}        
	\begin{figure}[h]
		\centering
		\includegraphics[width=0.7\textwidth]{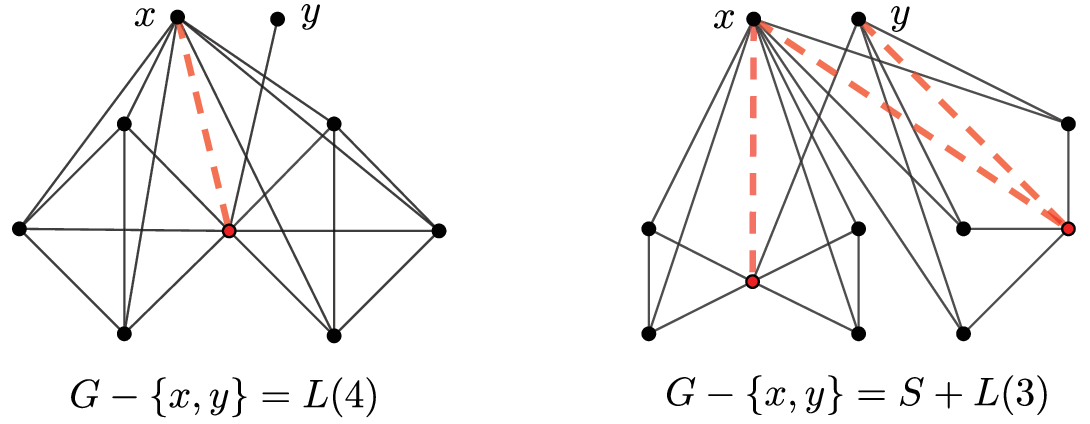}
		\caption{\footnotesize{The red dotted line is not necessarily present.}}
		\label{fig4}
	\end{figure}

	\begin{remark}\label{rm}
		By the definition of $L(\delta)$ and $S$, if either exists, then there is an $(x,y)$-path of order $\delta+2;$ note that $x$ and $y$ may be adjacent to the center of each star of $S$.
	\end{remark}

	\section{Proof of  Theorem \ref{thm1.3}}\label{sec4}
	
	Let $G$ be a 2-connected graph.
	We call graph $G$ \emph{edge-maximal} if for any edge $e\in E(\overline{G})$, $G+e$ will have a cycle of length at least $c(G)+1$. The following lemma will be very helpful for characterizing all extremal graphs in {\blue Theorem \ref{thm1.3}}.

	\begin{lemma}\label{l2.3}
		For two integers $w,\delta$ with $\omega\ge \delta\ge 2$,
		let $G$ be a $2$-connected edge-maximal non-hamiltonian graph of order $n$ with  $\delta(G)\ge \delta$, $\omega(G)\ge \omega$ and $c(G)=\omega+\delta$. Let $H\subseteq G$ be a complete graph of order $\omega(G)$, let $T=G[V(G)\setminus V(H)]$ and let $P=v_1v_2\dots v_m$ be a longest $(H,T)$-path with $v_1v_m\notin E(G)$. If $(d_P(v_1),d_P(v_m))= (\omega-1,\delta+1)$ or $(\omega,\delta)$, then
		\begin{wst}
			\item[{\rm(i)}] $G\in \{H_1(n,\omega,\delta+1),G_1,G_2\}$ when $(d_P(v_1),d_P(v_m))= (\omega-1,\delta+1);$ or
			\item[{\rm(ii)}] $G\in \{H_1(n,\omega+1,\delta),H_2(n,\omega+1,\delta),H_3(n,\omega,2),G_3,G_4\}$ when $(d_P(v_1),d_P(v_m))= (\omega,\delta).$
		\end{wst}
		
	\end{lemma}
	
	\begin{proof}
		
		Let $k=\omega+\delta+1.$
		By the edge-maximality of $G$ and $v_1v_m\notin E(G),$ $m\ge k.$ By the maximality of $P$, $N_H(v_1)\subseteq N_P(v_1)$ and $N_G(v_m)=N_P(v_m).$ By \Cref{L1} and $c(G)=k-1$, $P$ has a crossing pair.
		Let $s=\min\{i:v_iv_m\in E(G)\}$ and $t=\max\{j:v_jv_1\in E(G)\}.$ Clearly, $s\ge 2$ and $t\le m-1.$
		By $c(G)=k-1$,
		\begin{align}\label{a1}
			~N_P^-(v_1)\cap N_P[v_m]=\emptyset ~\text{and} ~N_P^+(v_m)\cap N_P[v_1]=\emptyset.
		\end{align}
		
		Let $(p,q)$ be a minimal crossing pair of $P$ with $s\le p<q\le t$ and let $C=v_1Pv_pv_mPv_qv_1$.
		Clearly, $N_P[v_1]\cup N_P[v_m]\subseteq V(C)$.
		Then $N_P^-(v_1)\cup N_P[v_m]\setminus\{v_{q-1}\}\subseteq V(C)$ and $N_P[v_1]\cup N_P^+(v_m)\setminus\{v_{p+1}\}\subseteq V(C).$ By $c(G)=k-1$, we have $|V(C)|\le k-1$.
		Since $(d_P(v_1),d_P(v_m))$ is either $(\omega-1,\delta+1)$ or $(\omega,\delta)$, by (\ref{a1}), we have
		$$k-1=|N_P^-(v_1)\cup N_P[v_m]\setminus\{v_{q-1}\}|=|N_P[v_1]\cup N_P^+(v_m)\setminus\{v_{p+1}\}|\leq|V(C)|\le k-1.$$
		Then
		\begin{align}\label{a}
			V(C)=N_P^-(v_1)\cup N_P[v_m]\setminus\{v_{q-1}\}=N_P[v_1]\cup N_P^+(v_m)\setminus\{v_{p+1}\}.
		\end{align}
		By (\ref{a}), it is readily seen that every minimal crossing pair is in fact a minimum crossing pair. Clearly, $|C|=k-1.$
		We now present more structural properties of $P$ in $G$.
		
		\begin{claim}\label{c1}
			Let $(d_P(v_1),d_P(v_m))= (\omega-1,\delta+1)$. Then  $N_P[v_1]=N_P[v_i]$ for $1\le i\le s-1$; and for $t+1\le j\le m$,  $N_P[v_j]\subseteq N_P[v_m]$ and $|[\{v_j\},N_P(v_m)]|\ge \delta$.
		\end{claim}

		\begin{proof}
			For $1\le i\le s-1$, $v_i\notin N_P(v_m)$. By (\ref{a}), $v_{i+1}\in N_P(v_1)$.
			Then $V(v_1Pv_s)\subseteq N_P[v_1]$, similarly, $V(v_tPv_{m})\subseteq N_P[v_m]$.
		
			We prove that $N_P[v_i]\subseteq N_P[v_1]$ for $1\le i\le s-1$. If $s=2$, we are done, so assume that $s\ge 3.$ To the contrary, suppose there exists a vertex $v_h\in N_P[v_i]\setminus N_P[v_1]$ for some $ 2\le i\le s-1$.
			Clearly, $h\geq s+1.$
			If $v_h\notin V(C)$, then $v_1Pv_iv_hPv_mv_pPv_{i+1}v_1$ is a cycle of length at least $k$, a contradiction.
			Hence, $v_h\in V(C)$. 
			By $v_h\notin N_P[v_1]$ and (\ref{a}), $v_h\in N^+_P(v_m)$, that is, $v_{h-1}\in N_P(v_m)$.
			But $v_1Pv_iv_hPv_mv_{h-1}Pv_{i+1}v_1$ is an $m$-cycle, a contradiction.
			Thus, $N_P[v_i]\subseteq N_P[v_1]$ for $1\le i\le s-1$.
			By $N_P(v_1)=N_H(v_1)$, $N_P[v_i]=N_P[v_1]$ for $1\le i\le s-1$.
			
			We prove that $N_P[v_j]\subseteq N_P[v_m]$ for $t+1\le j\le m$. If $t=m-1$, we are done, so assume that $t\le m-2$. To the contrary, suppose that there exists a vertex $v_h\in N_P[v_j]\setminus N_P[v_m]$ for some $ t+1\le j\le m-1$.
			Clearly, $h\leq t-1.$
			If $v_h\notin V(C)$, then $v_1Pv_hv_jPv_mv_{j-1}Pv_qv_1$ is a cycle of length at least $k$, a contradiction.
			Hence, $v_h\in V(C)$. 
			By $v_h\notin N_P[v_m]$ and (\ref{a}), $v_{h+1}\in N_P(v_1)$. But $v_1Pv_hv_jPv_mv_{j-1}Pv_{h+1}v_1$ is an $m$-cycle, a contradiction. Thus, $N_P[v_j]\subseteq N_P[v_m]$ for $t+1\le j\le m$. Consider the $(H,T)$-path $P'=v_1Pv_{j-1}v_mPv_{j}$ with $t+1\le j\le m$. We have $\delta\le d_G(v_{j})=d_{P'}(v_{j})=d_P(v_{j})\le \delta+1$.
			If $d_P(v_j)=d_P(v_m)=\delta+1$, then $N_P[v_j]=N_P[v_m]$. If $d_P(v_j)=\delta$, then $|[\{v_j\},N_P(v_m)]|\ge \delta$.
			This proves \Cref{c1}.
		\end{proof}
		
		\begin{claim}\label{c2}
			Let $(d_P(v_1),d_P(v_m))= (\omega,\delta)$. Then $N_P[v_i]\subseteq N_P[v_1]$ for $1\le i\le s-1$ and $N_P[v_j]= N_P[v_m]$ for $t+1\le j\le m$; moreover, if $v_i\in V(H)$, then $|[\{v_i\},N_P(v_1)]|\ge \omega-1.$
		\end{claim}
		\begin{proof}
			Similar to the proof of \Cref{c1}, $N_P[v_i]\subseteq N_P[v_1]$ for $1\le i\le s-1$ and $N_P[v_j]\subseteq N_P[v_m]$ for $t+1\le j\le m$.
			Moreover, for $1\le i\le s-1$, if $v_i\in V(H)$, then $d_P(v_i)\ge \omega-1$ and $|[\{v_i\},N_P(v_1)]|\ge \omega-1.$
			Consider the $(H,T)$-path $P''=v_1Pv_{j-1}v_mPv_{j}$, for $t+1\le j\le m$. We have $d_G(v_{j})=d_{P''}(v_{j})=d_{P}(v_{j})= \delta$.
			Then $N_P[v_{j}]=N_P[v_m]$ for $t+1\le j\le m$.
			This proves \Cref{c2}.
		\end{proof}
		
		By the proof of \Cref{c1} and \Cref{c2}, $N_G[v_j]=N_P[v_j]$ for $t+1\le j\le m$.
		We need to define some sets. Let $A_1=V(v_1Pv_{s-1}),  A_2=N_P(v_1)\setminus V(v_2Pv_{s-1}),
		A_2'=N_P(v_m)\setminus V(v_{t+1}Pv_{m-1}), A_3=V(v_{t+1}Pv_{m})$ and let
		$X=V(G)\setminus (A_1\cup A_2\cup A_3)$. Clearly, $X\ne\emptyset.$

		\begin{claim}\label{c3}
			If $v_1$ and $v_m$ are not adjacent to any two consecutive vertices of  $v_sPv_t$,  then $A_2=A_2'=\{v_s,v_t\}$ when $(s,t)$ is a minimal crossing pair or  $A_2=A_2'=\{v_s,v_{s+2},\dotsc, v_{t-2},v_t\}$ where $t\equiv s\pmod 2$ and $m=k$ when $(s,t)$ is not a minimal crossing pair.
		\end{claim}
		
		\begin{proof}
			If $(s,t)$ is a minimal crossing pair, then $A_2=A_2'=\{v_s,v_t\}$ by \Cref{c1} and \Cref{c2}.
			Suppose that $(s,t)$ is not a minimal crossing pair.
			Let $(i,j)$ be a minimal crossing pair.
			Clearly, $i> s$ or $j< t$.
			Without loss of generality, let $i>s$.
			
			Since $v_m$ is not adjacent to any two consecutive vertices of $v_sPv_t$, we have $v_{i-1}v_m\notin E(G).$
			By (\ref{a}) and $v_i\notin N^+_P(v_m)$, $v_iv_1\in E(G)$.
			Since $v_1$ is not adjacent to any two consecutive vertices of $v_sPv_t$, we have $v_{i-1}v_1\notin E(G)$.
			By (\ref{a}) and $v_{i-2}\notin N^-_P(v_1)$, $v_{i-2}v_m\in E(G)$.
			Since $(i-2,i)$ is a minimal crossing pair, we have $\left|V(v_iPv_j)\right|=3$ by the choice of $(i,j).$
			Repeating the above arguments, we have $\{v_s,v_{s+2},\dots,v_{i-2},v_i\}=A_2\cap V(v_sPv_i)= A_2'\cap V(v_sPv_i)$ and $i\equiv s\pmod 2.$
			
			Similarly, $v_1v_{j+1}\notin E(G)$. By (\ref{a}) and $v_{j}\notin N^-_P(v_1),$ $v_{j}v_m\in E(G).$ Hence, $v_{j+1}v_m\notin E(G).$ By (\ref{a}) and $v_{j+2}\notin N^+_P(v_m)$, $v_1v_{j+2}\in E(G).$ Repeating the above arguments, we have $\{v_{j},v_{j+2},\dots,v_{t-2},v_{t}\}= A_2\cap V(v_jPv_t)= A_2'\cap V(v_jPv_t)$ and $t\equiv j\pmod 2.$
			
			Base on the fact that $\left|V(v_iPv_j)\right|=3$, $A_2=A_2'=\{v_s,v_{s+2},\dotsc,v_{t-2},v_t\}$ and $t\equiv s\pmod 2$. Moreover, $m=k$.
			This proves \Cref{c3}.
		\end{proof}
		
		By analyzing the structure of $G$ when $(d_P(v_1),d_P(v_m))=(\omega-1,\delta+1)$, we obtain  the following two claims.
		
		\begin{claim}\label{c4}
			If $(d_P(v_1),d_P(v_m))=(\omega-1,\delta+1)$, then $v_1$ and $v_m$ are not adjacent to any two consecutive vertices of  $v_sPv_t$.
		\end{claim}
		
		\begin{proof}
			To the contrary, suppose $v_i,v_{i+1}\in V(v_sPv_t)$ and $v_i,v_{i+1}\in N_P(v_1).$
			Since $c(G)=k-1$, we have $i\geq s+2.$
			By \Cref{c1}, $v_{i+1}v_{s-1}\in E(G).$
			Then $v_1Pv_{s-1}v_{i+1}Pv_mv_sPv_iv_1$ is a cycle of length $m\ge k$,  a contradiction.
			
			Suppose $v_i,v_{i+1}\in N_P(v_m).$
			Since $c(G)=k-1$, $i+1\leq t-2.$
			By \Cref{c1}, $v_{t+1}v_i\in E(G)$ or $v_{t+1}v_{i+1}\in E(G)$.
			If $v_{t+1}v_i\in E(G)$, then $v_1Pv_iv_{t+1}Pv_mv_{i+1}Pv_tv_1$ is a cycle of length  $m\ge k$, a contradiction.
			If $v_{t+1}v_{i+1}\in E(G)$, then $v_1Pv_iv_mPv_{t+1}v_{i+1}Pv_tv_1$ is a cycle of length  $m\ge k$, a contradiction.
			This proves \Cref{c4}.
		\end{proof}
		
		\begin{claim}\label{c7'}
			If $(d_P(v_1),d_P(v_m))=(\omega-1,\delta+1)$, then each vertex of $X$ can only be connected to $A_2\subseteq A_1\cup A_2\cup A_3$; moreover, if $(s,t)$ is not a minimal crossing pair, then $X$ is an independent set of $G$.
		\end{claim}
		
		\begin{proof}
			By \Cref{c1}, \Cref{c3} and \Cref{c4}, $N_P[A_1]=A_1\cup A_2$ and $N_G[A_3]\subseteq A_2\cup A_3$, that is, $[X\cap V(P), A_1]=\emptyset$ and $[X, A_3]=\emptyset$. 
			First, we show that each vertex of $X\setminus V(P)$ is only connected to $A_2\subseteq A_1\cup A_2\cup A_3$. 
			To the contrary, suppose $z\in X\setminus V(P)$ is connected to $A_1\cup A_2\subseteq A_1\cup A_2\cup A_3$.
			Since $G$ is 2-connected, there exist two vertex-disjoint $(z,P)$-paths $P_1$ and $P_2$ with $V(P)\cap V(P_i)=x_i$ for $i=1,2$. Assume that $x_1\in V(v_1Px_2).$ By the maximality of $P$, $x_1x_2\notin E(P)$.
			
			Suppose $x_1,x_2\in A_1$. By $V(H)= N_P[v_1]$, $x_1^-x_2^-\in E(G)$ or $x_1^+x_2^+\in E(G)$.
			If $x_1^-x_2^-\in E(G)$, then $v_1Px_1^-x_2^-Px_1P_1zP_2x_2Pv_m $ is an $(H,T)$-path longer than $P$, a contradiction.
			If $x_1^+x_2^+\in E(G)$, then $v_1Px_1P_1zP_2x_2Px_1^+x_2^+Pv_m$ is an $(H,T)$-path longer than $P$, a contradiction.
			Thus, $x_1\notin A_1$ or $x_2\notin A_1$.
			Without loss of  generality,  let $(s,t)$ be a minimal crossing  pair. Recall that $v_1Pv_sv_mPv_tv_1$ is a $(k-1)$-cycle.
			Another situation is similar. If $x_1\in A_1$ or $x_2\in X\cap V(P)$, then  $v_1Px_1P_1zP_2x_2Pv_mv_sPx_1^+v_1$ is a cycle of length at least $k$, a contradiction.
			Suppose $x_1\in A_1$ and $x_2\in A_2$. If $x_2=v_t$, then $v_1Px_1P_1zP_2x_2Pv_mv_sPx_1^+v_1$ is a cycle of length at least $k$, a contradiction.
			Hence, $x_2=v_s$.
			By $x_1x_2\notin E(P)$, $x_1\ne v_{s-1}$ and then $v_1Px_1P_1zP_2x_2v_mPv_tv_{s-1}Px_1^+v_1$ is a cycle of length at least $k$, a contradiction.
			
			By $[X\cap V(P), A_1\cup A_3]=\emptyset$, each vertex of $X\cap V(P)$ is only connected to $A_2\subseteq A_1\cup A_2\cup A_3$.
			Thus, each vertex of $X$ can only be connected to $A_2\subseteq A_1\cup A_2\cup A_3$. 
			
			Assume that $(s,t)$ is not a minimal crossing pair. We show that $X$ is an independent set of $G$.
			By \Cref{c3}, $A_2=A_2'=\{v_s,v_{s+2},\dotsc,v_{t-2},v_t\}.$
			For any two vertices $y_1,y_2\in X\cap V(P)$, there exists a $(y_1,y_2)$-path of order $k$ (see Figure \ref{fig5}(a)).
			Hence, $X\cap V(P)$ is an independent set of $G$.
			For any two vertices $y_1,y_2\in A_2$, there exists a $(y_1,y_2)$-path of order $k-2$ (see Figure \ref{fig5}(b)); for any two vertices $y_1\in A_2$ and $y_2\in X\cap V(P)$, there exists a $(y_1,y_2)$-path of order $k-1$ (see Figure \ref{fig5}(c)).
			Since $c(G)=k-1$ and $G$ is 2-connected, $X\setminus V(P)$ is an independent set of $G$ and $[X\cap V(P), X\setminus V(P)]=\emptyset$.
			Thus, $X$ is an independent set of $G$. This proves \Cref{c7'}.
		\end{proof}
		
		\begin{figure}[h]
			\centering
			\includegraphics[width=1\textwidth]{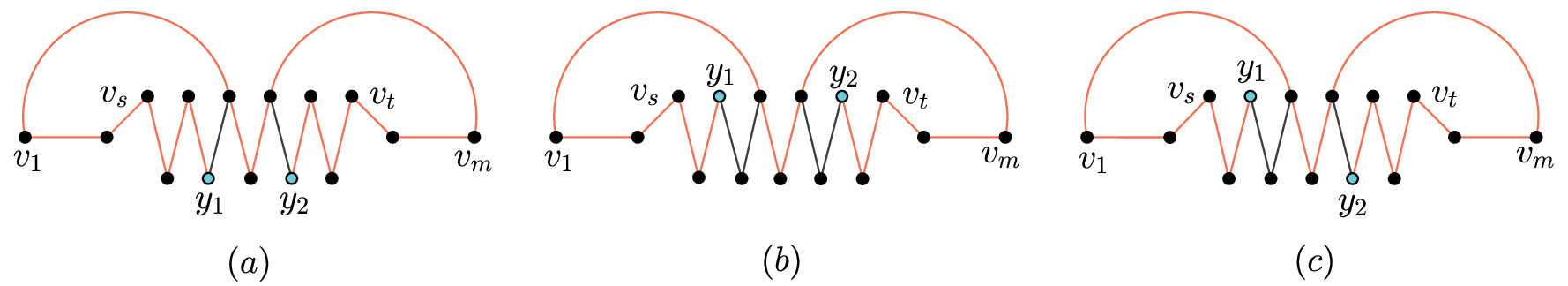}
			\caption{\footnotesize{The red path is a $(y_1, y_2)$-path we identified.}}
			\label{fig5}
		\end{figure}
		
		Our analysis of the structure of $G$ under the condition $(d_P(v_1),d_P(v_m))=(\omega,\delta)$ leads to the following four claims.
		
		\begin{claim}\label{cc5}
			If $(d_P(v_1),d_P(v_m))=(\omega,\delta)$, then $v_m$ is not adjacent to any two consecutive vertices of  $v_sPv_t$; moreover, if $v_1$ is adjacent to two consecutive vertices of $v_sPv_t$, then $N_P[v_1]\setminus V(H)=\{v_2\}=\{v_{s-1}\}$.
		\end{claim}

		\begin{proof}
			Let $v_j,v_{j+1}\in V(v_sPv_t)$.
			Suppose $v_j,v_{j+1}\in N_P(v_m).$ Then $j+1\le t-2$. Similar to the proof of \Cref{c4}, it is easy to find an $m$-cycle, a contradiction. Assume that   $v_j,v_{j+1}\in N_P(v_1)\cap V(v_sPv_t).$ Clearly, $j\ge s+2.$ 
			
			Suppose $v_2\notin N_P[v_1]\setminus V(H).$
			By \Cref{c2}, $v_2v_j\in E(G)$ or $v_2v_{j+1}\in E(G)$. 
			If $v_{2}v_j\in E(G)$, then
			$v_1v_{s-1}Pv_2v_jPv_sv_mPv_{j+1}v_1$ is an $m$-cycle. If $v_2v_{j+1}\in E(G)$, then $v_1v_{s-1}Pv_2v_{j+1}\allowbreak Pv_mv_sPv_jv_1$ is an $m$-cycle. Both situations lead to a contradiction. 
			
			Suppose $v_2\in N_P[v_1]\setminus V(H)$ and $v_2\ne v_{s-1}$. By \Cref{c2}, $v_{s-1}v_j\in E(G)$ or $v_{s-1}v_{j+1}\in E(G)$. 
			If $v_{s-1}v_j\in E(G)$, then $v_1Pv_{s-1}v_jPv_sv_mPv_{j+1}v_1$ is an $m$-cycle. If $v_{s-1}v_{j+1}\in E(G)$, then $v_1Pv_{s-1}v_{j+1}Pv_mv_sPv_jv_1$ is an $m$-cycle. Both situations lead to a contradiction.
			
			Thus, if $v_1$ is adjacent to  two consecutive vertices of $v_sPv_t$, then $N_P[v_1]\setminus V(H)=\{v_2\}=\{v_{s-1}\}$.
			This proves \cref{cc5}.
		\end{proof}
		
		\begin{claim}\label{c5}
			If $(d_P(v_1),d_P(v_m))=(\omega,\delta)$ and $v_1$ is adjacent to two consecutive vertices of $v_sPv_t$, then $G=G_4$.
		\end{claim}
		\begin{proof}
			By \cref{cc5}, $v_2\in N_P[v_1]\setminus V(H)$ and $s=3$.
			By $d_P(v_1)=\omega$,  we have $|H|=\omega(G)=\omega.$
			
			Suppose that $v_3Pv_t$ contains $h$ vertex-disjoint segments $v_{i_1}Pv_{j_1},v_{i_2}Pv_{j_2},\dots,v_{i_h}Pv_{j_h}$ with $i_1< i_2 < \dots< i_h,$ where each segment is a subpath of order at least two and $V(v_{i_f}Pv_{j_f})\subseteq N_P[v_1]$ for $1\le f\le h$. Without loss of generality, assume that each segment contains as many vertices as possible. Next, we show that $h=1$.
			
			By $v_3v_m,v_{i_{1}+1}v_1\in E(G)$ and $c(G)=k-1$, $i_1\ge 5$. Note that there exists a minimal crossing pair $(p_1,q_1)$ of $P$ with $3\le p_1<q_1\le i_1$.  Similar to the proof of \Cref{c3}, we obtain $$N_P(v_1)\cap V(v_3Pv_{i_1})=N_P(v_m)\cap V(v_3Pv_{i_1})=\{v_3,v_5,\dots,v_{i_1-4},v_{i_1-2}\}$$ where $i_1\equiv 3\pmod 2$. 
			For $1\le f\le h$,
			a similar analysis of the subpath 
			$v_{j_f}Pv_{i_{f+1}}$(suppose $i_{h+1}=t$) shows that $$N_P(v_1)\cap V(v_{j_f}Pv_{i_{f+1}})=N_P(v_m)\cap V(v_{j_f}Pv_{i_{f+1}})=\{v_{j_{f}},v_{j_{f+2}},\dots,v_{i_{f+1}-4},v_{i_{f+1}-2}\}$$ where $i_{f+1}\equiv j_f\pmod 2$.
			Hence, for any vertex $x\in V(v_3Pv_t)\setminus N_P(v_1)$, $x^+,x^-\in N_P[v_1]$.
			Let $W=V(v_3Pv_t)$. By $c(G)=k-1$, we have
			$$N_P(v_m)\cap W=(N_P(v_1)\cap W)\setminus (\mathop{\cup}\limits_{f=1}^{h}V(v_{i_f}Pv_{j_f^-})).$$
			Consider the $(H,T)$-path $v^+_{i_1}Pv_mv_3Pv_{i_1}v_1v_2$. By the choice of $P$ and \Cref{c2}, $N_G(v_2)=N_P(v_2)\subseteq N_P[v_1]$. 
	
			Clearly, $[\{v_2\},V(v_{i_f}Pv_{j_f})]=\emptyset$ for $1\le f\le h$. Hence, 
			$$N_P(v_2)\cap W\subseteq (N_P(v_m)\cap W)\setminus \{v_{j_1},v_{j_2},\dots,v_{j_h}\}.$$ 
			Let $N_W(v_2)=N_P(v_2)\cap W$ and $N_W(v_m)=N_P(v_m)\cap W$. Then $$\delta-1\le |N_G(v_2)|-1=|N_W(v_2)|\le |N_W(v_m)|-h= |N_G(v_m)|-h-(m-t-1)=\delta-h-(m-t-1).$$ By $m\ge t+1$ and $h\ge 1$, $h=1$ and $m=t+1$. Each inequality becomes equality. 
			Thus, $$N_P(v_2)\cap W=(N_P(v_m)\cap W)\setminus \{v_{j_1}\}.$$
			By $|N_W(v_m)|=|N_W(v_2)|+1=d_P(v_2)\ge \delta$, $j_1=t.$
			Recall that $v_{i_1}\ne v_3$.
			Then $N_P(v_1)\cap W= N_P(v_m)\cap W=\{v_3,v_5,\dots,v_{i_1-4},v_{i_1-2},v_t\}$ where $i_1\equiv 3\pmod 2$.
			Since $v_1v_2v_3v_mPv_5v_1$ is an $(m-1)$-cycle, we have $m=k$.

			We show that $V(G)\setminus V(H)$ is an independent set. Recall that for any vertex $x\in W\setminus N_P(v_1)$, then $x^+,x^-\in N_P[v_1]$. For any two vertices $x,y \in V(P)\setminus N_P[v_1]$, there exists an $(x,y)$-path of order $k$. Thus, $V(P)\setminus N_P[v_1]$ is an independent set. For any two vertices $u,v \in  N_P[v_1]\cap W$, there exists a $(u,v)$-path of order at least $k-2$. Since $G$ is $2$-connected, there exist two vertex-disjoint $(z,P)$-paths $P_1$ and $P_2$ with $V(P)\cap V(P_i)=x_i$ for $i=1,2$ and $z\in V(G)\setminus V(P)$. Clearly, $x_1,x_2\in N_P[v_1]\cap W.$ Hence, $P_i$ is an edge with $i=1,2$. Thus, $V(G)\setminus V(P)$ is an independent set and then  $V(G)\setminus V(H)$ is an independent set.
			
			We assert that $i_1=5$.  Suppose $i_1\ne 5$. 
			Since $V(G)\setminus V(H)$ is an independent set, we have $N_G(x)\subseteq N_W(v_m)$ for any vertex $x\in W\setminus N_P(v_1).$
			By $d_P(v_{i_1-3})\ge \delta$, $v_{i_1-3}v_{j_1}\in E(G)$. Then $v_{i_1-2}v_mv_{j_1}v_{i_1-3}Pv_2v_1v_{j_1-1}Pv_{i_1}v_{i_1-1}v_{i_1-2}$ is a $k$-cycle, a contradiction.

			Now $d_G(v_m)=\delta=2.$
			From the above, we determine the neighborhood of the vertices in $V(P)\setminus V(H)$ within the graph $G$ (see Figure \ref{fig6}).
			
			\begin{figure}[h]
				\centering
				\includegraphics[width=0.5\textwidth]{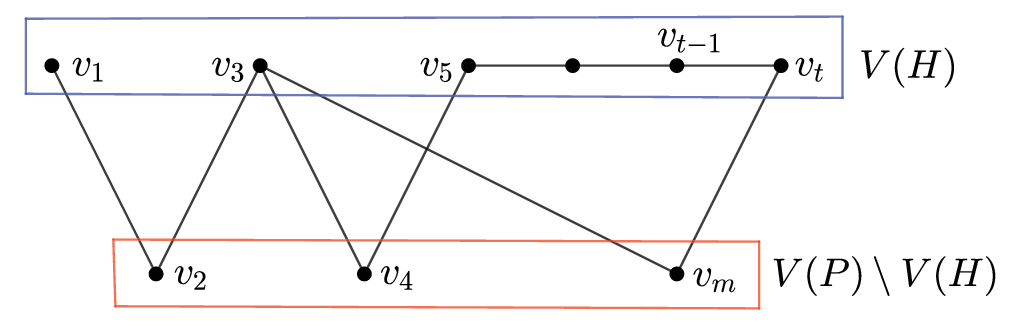}
				\caption{\footnotesize{The neighborhood of the vertices in $V(P)\setminus V(H)$ within the graph $G$.}}
				\label{fig6}
			\end{figure}
			
			Let $w$ be any vertex in $V(G)\setminus V(P)$. Clearly, $N(w)\subseteq \{v_1,v_3,v_5,v_t\}.$
			We assert that $N(w)=\{v_3,v_5\},\{v_3,v_t\}~\text{or}~\{v_1,v_3\}$. Suppose not. 
			Without loss of generality, suppose $\{v_5,v_t\}\subseteq N(w)$. Then $v_1Pv_5wv_tPv_6Pv_{t-1}v_1$ is a cycle of length $k,$ a contradiction. 
			
			Hence, $\delta=2$, $\omega\ge 4$ and $G=G_4$. This proves \Cref{c5}.
		\end{proof}

		\begin{claim}\label{c6'}
			If  $(d_P(v_1),d_P(v_m))=(\omega,2)$ and $v_1$ is not adjacent to any two consecutive vertices of $v_sPv_t$, then
			$G\in  \{H_1(n,\omega+1,2), H_3(n,\omega,2),G_3\}.$
		\end{claim}
		\begin{proof}
			By $\delta=2$, $t=m-1$ and $N_P(v_m)=\{v_s,v_{m-1}\}$.
			By \Cref{c3} and $c(G)=k-1$, $t-s=2$ and $(s,t)$ is a minimal crossing pair. Clearly, $X\cap V(P)=\{v_{t-1}\}$. By \Cref{c2}, $N_P[A_1]\subseteq A_1\cup A_2$ and $N_G(v_m)\subseteq A_2.$
			Suppose $X\setminus V(P)= \emptyset$, that is, $|X|=1.$ By the maximality of $G$, $G=K_2\vee(K_{\omega-1}+2K_1)=
			H_1(n,\omega+1,2)$.
			Assume that $X\setminus V(P)\neq \emptyset$. Consider the path $v_1Pv_sv_mv_tv_{t-1}. $ Then $[X\cap V(P),X\setminus V(P)]=\emptyset$ and then $N_G(v_{t-1})=\{v_s,v_t\}.$
			
			First, we show that $X$ is an independent set of $G$. To the contrary, suppose $x_1,x_2\in X\setminus V(P)$ with $x_1x_2\in E(G).$
			Since $G$ is 2-connected, there are two vertex-disjoint paths $x_1P_1y_1$ and $x_2P_2y_2$ where $V(P_i)\cap V(P)=\{y_i\}$ for $i=1,2$. Assume that $y_1\in V(v_1Py_2).$ By the maximality of $P$, $|V(y_1Py_2)|\ge 4.$
			Since $N_P(v_m)=\{v_s,v_t\}$ and $[X\cap V(P),X\setminus V(P)]=\emptyset$, we have $y_1,y_2\in  A_1\cup \{v_s\}.$
			Without loss of generality, let $y_1\in A_1$.
			Suppose $y_2=v_s$.
			If $y_2^-\in V(H)$, then $y_2^-Py_1^+v_1Py_1P_1x_1x_2P_2v_sPv_m$ is an $(H,T)$-path longer than $P$, a contradiction.
			If $y_2^-\notin V(H)$, then $y_1^+\in V(H)$ and then $y_1^+Py_2^-v_1Py_1P_1x_1x_2P_2v_sPv_m$ is an $(H,T)$-path longer than $P$, a contradiction.
			Hence, $y_2\in  A_1$. 
			If $y_1^+,y_2^+\in V(H)$, then  $v_1Py_1P_1x_1x_2P_2y_2Py_1^+y_2^+Pv_m$ is an $(H,T)$-path longer than $P$, a contradiction.
			Hence, $y_1^+\in N_P[v_1]\setminus V(H)$ or $y_2^+\in N_P[v_1]\setminus V(H)$.
			Then $y_2^-\in V(H)$ and then $y_2^-Pv_1P_1x_1x_2P_2y_2Pv_m$ when $v_1=y_1$, or $v_1Py_1^-y_2^-Py_1P_1x_1x_2P_2y_2Pv_m$  when $v_1\ne y_1$ is an $(H,T)$-path longer than $P$, a contradiction. Thus, $X$ is an independent set of $G$.
			
			If each vertex in $X\setminus V(P)$ is only connected to $A_2\subseteq A_1\cup A_2\cup A_3$, by the maximality of $G$, then $G=H_1(n,\omega+1,2).$
			Suppose there is a vertex $x_0\in X\setminus V(P)$ that is connected to $A_1\cup A_2\subseteq A_1\cup A_2\cup A_3$.
			Since $G$ is 2-connected, there are two vertex-disjoint $(x_0,P)$-paths $P_1$ and $P_2$ with $V(P_i)\cap V(P)=\{z_i\}$ for $i=1,2$. Since $X$ is an independent set, $P_i$ is an edge for $i=1,2.$ Assume that $z_1\in V(v_1Pz_2)$ and $(z_1,z_2)\neq (v_s,v_{t}).$ Clearly, $|V(z_1Pz_2)|\ge3$.
			
			We assert that $z_1=v_1$. Suppose not. Since $z_1^-Pv_1z_2^-Pz_1x_0\allowbreak z_2Pv_m$ is a path longer than $P$, we have $ N_P(v_1)\setminus V(H)=\{z_1^-\}$, and hence $z_1^+\in V(H)$. Now $z_1^+Pz_2^-v_1Pz_1x_0z_2Pv_m$ is an $(H,T)$-path longer than $P,$ a contradiction.
			
			We assert that $|V(z_1Pz_2)|=3$. Suppose $|V(z_1Pz_2)|\ge 4$.
			We show that $z_2\in A_1.$ Otherwise, supppose $z_2\notin A_1.$ If $z_2=v_t$, then $v_1x_0v_tv_mv_sPz_1^+v_1$ is a cycle of length $m$, a contradiction.
			Thus, $z_2=v_s$. Since $z_1^+Pz_2^-v_1\allowbreak x_0z_2Pv_m$ is a path than $P$, we have $N_P(v_1)\setminus V(H)=\{z_1^+\}.$ By $|V(z_1Pz_2)|\ge 4$, $z_2^-\in V(H)$.
			Now $z_2^-Pv_1x_0v_sPv_m$ is an $(H,T)$-path longer than $P,$ a contradiction. Thus, $z_2\in A_1.$ 
			By $c(G)=k-1$, $z_1^+z_2^+\notin E(G)$ and then  $z_2^-\in V(H)$. Now $z_2^-Pv_1x_0z_2Pv_m$ is an $(H,T)$-path longer than $P,$ a contradiction. Thus, $|V(z_1Pz_2)|=3$.
			
			By $z_1=v_1$ and $|V(z_1Pz_2)|=3$, we have $z_2\in A_1\cup \{v_s\}.$ Suppose $z_2\in A_1$. Then $N_P(v_1)\setminus V(H)=\{z_1^+\}$ and hence $\omega(G)=\omega.$
			We assert that $d(z_1^+)=2$. Suppose $d(z_1^+)\ge 3$. By \Cref{c2}, $N(z_1^+)\subseteq V(H)$. It is easy to verify that there is a cycle of length $m$, a contradiction. 
			By $z_2\ne v_s$, $\omega(G)=\omega\ge 4$ and hence $G=H_3(n,\omega,2)$. Suppose $z_2=v_s.$ Then $\omega(G)=\omega=3$ and $\delta=2.$ 
			Thus, $G=G_3$ where $l_1\ge 1$, $l_2=0$, $l_3\ge 2$, $|S|=0$ and $|L(\delta)|=0$. 
			This proves \Cref{c6'}.
		\end{proof}
		
		By \Cref{c5} and \Cref{c6'}, in the case where $(d_P(v_1),d_P(v_m))=(\omega,\delta)$, we may assume  $\delta\ge 3$ and that $v_1$ is not adjacent to any two consecutive vertices of $v_sPv_t$ in $G.$ 
		
		\begin{claim}\label{c8'}
			For $\delta\ge 3$, if $(d_P(v_1),d_P(v_m))=(\omega,\delta)$ and $v_1$ is not adjacent to any two consecutive vertices of $v_sPv_t$, then each vertex of $X$ can only be connected to $A_2\subseteq A_1\cup A_2\cup A_3$; moreover, $X$ is an independent set of $G$ when $(s,t)$ is not a minimal crossing pair.
		\end{claim}
		\begin{proof}
			By \Cref{c2}, \Cref{c3},  \Cref{cc5} and \Cref{c5}, we have $N_P[A_1]\subseteq A_1\cup A_2$ and $N_G[A_3]\subseteq A_2\cup A_3.$  First, 
			we show that each vertex of $X\setminus V(P)$ can only be connected to $A_2\subseteq A_1\cup A_2\cup A_3$.
			
			To the contrary, suppose there exists a vertex $z\in X\setminus V(P)$ that is connected to $A_1\cup A_2\subseteq A_1\cup A_2\cup A_3$.
			Since $G$ is 2-connected, there exist two vertex-disjoint $(z,P)$-paths $P_1$ and $P_2$ with $V(P)\cap V(P_i)=x_i$ for $i=1,2$. Assume that $x_1\in V(v_1Px_2).$ By the maximality of $P$, $x_1x_2\notin E(P)$.
			Similar to the proof of \Cref{c7'}, if $x_1\notin A_1$ or $x_2\notin A_2$, then we can find a cycle of length at least $k$; if $x_1^-x_2^-\in E(G)$ or $x_1^+x_2^+\in E(G)$, then there exists an $(H,T)$-path longer than $P$.
			Both situations lead to a contradiction.
			
			Hence, $x_1,x_2\in A_1$ and $x_1^-x_2^-, x_1^+x_2^+\notin E(G)$.
			Suppose $x_1=v_1$. Now $x_2^-Pv_1P_1zP_2x_2Pv_m$ is a path longer than $P$. By the choice of $P$, $x_2^-\in N_P[v_1]\setminus V(H).$ By $x_1^+x_2^+\notin E(G)$, we have $x_1^+=x_2^-\in N_P[v_1]\setminus V(H)$ and hence $x_2=v_3$.  By $\delta\ge 3$ and \Cref{c2}, there is a vertex $y\in N_P(v_1)\setminus \{v_3\}$ such that $yx_1^+\in E(G)$. If $y\in A_1$, then there is an $(H,T)$-path longer than $P$, a contradiction. If $y\in A_2$, then there is a cycle of length at least $k$, a  contradiction.
			Hence, $x_1\ne v_1$. By  $x_1^-x_2^-, x_1^+x_2^+\notin E(G)$, we have $x_1^+=x_2^-\in N_P[v_1]\setminus V(H).$ 
			By \Cref{c2}, $v_1x_1^+\in E(G)$ and then $x_1^-Pv_1x_1^+x_1P_1zP_2x_2Pv_m$ is an $(H,T)$-path longer than $P$, a contradiction.
			
			By $[X\cap V(P),A_1\cup A_3]=\emptyset$, each vertex of $X\cap V(P)$ can only be connected to $A_2\subseteq A_1\cup A_2\cup A_3$.
			Thus, each vertex of $X$ is only connected to $A_2\subseteq A_1\cup A_2\cup A_3$.
			An argument similar to the proof of \Cref{c7'} shows that if $(s,t)$ is not a minimal crossing pair,
			then $X$ is an independent set of $G$.
			This proves \Cref{c8'}.
		\end{proof}

		Next, we consider the following two cases.
		\vskip 3mm
		{\bf Case 1.} $(d_P(v_1),d_P(v_m))=(\omega-1,\delta+1).$
		\vskip3mm
		Since $N_H(v_1)\subseteq N_P(v_1)$ and $N_G(v_m)=N_P(v_m),$ we have $\omega(G)=\omega$ and $\delta(G)\ge \delta$. Recall that $X\ne \emptyset.$
		
		{\bf Subcase 1.1.}
		$(s,t)$ is not a minimal crossing pair.
		
		By \Cref{c3}, $A_2=A_2'=\{v_s,v_{s+2},\cdots, v_{t-2},v_t\}.$
		By \Cref{c7'}, $X$ is an independent set of $G$. And then $|A_2|\ge \delta$ since $\delta(G)\geq \delta$.
		By $|A_2\cup A_3|=\delta+2$, $(|A_2|,|A_3|)$ is either $(\delta,2)$ or $(\delta+1,1).$
		
		Suppose $(|A_2|,|A_3|)=(\delta,2)$.
		By $A_1\cup A_2=N_P[v_1],$ $|A_1|=\omega-\delta.$ 
		Now $G$ is the graph with a vertex partition $A_1\cup A_2\cup A_3\cup X$ with sizes $\omega-\delta$, $\delta$, $2$ and $n-\omega-2$, respectively, whose edge set consists of $[X, A_2]$, all edges in $A_1\cup A_2$, and edges in $A_2\cup A_3$. Note that $\delta(G)=\delta$ and $\omega(G)\ge \delta+1$. Then $\omega=\omega(G) \ge\delta+1.$
		Since $G$ is non-hamiltonian,  $n-\omega-2+1\ge \delta,$ i.e., $n\ge \omega+\delta+1.$
		By the maximality of $G$, $G= G_1$.
		
		Suppose $(|A_2|,|A_3|)=(\delta+1,1)$.
		Now $G$ is the graph with a vertex partition $A_1\cup A_2\cup A_3\cup X$ with sizes $\omega-\delta-1$, $\delta+1$, $1$, and $n-\omega-1$,  respectively, whose edge set consists of $[X\cup A_3, A_2]$, and all edges in $A_1\cup A_2$.
		Note that $\omega=\omega(G)\ge \delta+2$. Then $\omega\ge \delta+2.$
		Since $G$ is non-hamiltonian,  $n-\omega-1+1\ge \delta+1,$ i.e., $n\ge \omega+\delta+1.$
		By the maximality of $G$, $G= H_1(n,\omega,\delta+1)$. 
		
		{\bf Subcase 1.2.}
		$(s,t)$ is a minimal crossing pair.
		
		Since $G$ is 2-connected, by \Cref{c7'}, $G[X\cup A_2]$ is 2-connected.
		By \Cref{c3}, $A_2=A_2'=\{v_s,v_t\}$, $|A_1|=\omega-2$ and $|A_3|=\delta$.
		By the maximality of $G$ and \Cref{c7'}, $G[A_2\cup A_3]=K_{\delta+2}$.
		Since $c(G)=k-1,$ $G[A_2\cup X]$ contains no $(v_s,v_t)$-path of order more than $\min\{|A_1|+2,|A_3|+2\}$. By $\delta(G)\ge \delta$, we have $|X|\ge \delta-1.$
		Based on the relationship between the sizes of $A_1$ and $A_3$, we proceed to consider the following two cases.
		
		\item[$\bullet$]~ $|A_1|< |A_3|.$

		Suppose $\omega=\delta$. By \Cref{c7'}, each vertex of $X$ is only connected to $A_2\subseteq A_1\cup A_2 \cup A_3$. Hence, $A_1=\emptyset.$ Otherwise, 
		each vertex of $A_1$ has degree $\omega-1=\delta-1$, a contradiction. By $|A_1|=\omega-2,$ we have $\omega=\delta=2.$ Then there exists no $(v_s,v_t)$-path of order more than two in $G[A_2\cup X]$, contradicting $|X|\ge \delta-1=1$.  
		
		Thus, $\omega = \delta+1.$ There exists no $(v_s,v_t)$-path of order more than $\delta+1$ in $G[A_2\cup X].$
		Recall that $|X|\ge \delta-1.$ If $|X|=\delta-1,$ then $A_2\cup X$ is a $(\delta+1)$-clique since $\delta(G)\ge \delta$. Now $G=G_2$ where $l_1=1,l_2=1$, $|S|=0$ and $|L(\delta)
		|=0.$ Suppose $|X|=\delta.$ Since there exist no $(v_s,v_t)$-path of order more than $\delta+1$ in $G[A_2\cup X]$, by \Cref{L2.2},
		we have $\delta=2$ and $X=\overline{K}_2$.
		Hence, $G=K_2\vee(K_1+K_2+\overline{K}_2)=G_2$ where $l_1=2$, $l_2=1$, $|S|=0$ and $|L(\delta)|=0$.
		Suppose $|X|\ge \delta+1$.
		Now $(G[A_2\cup X],v_s,v_t)$ is a $2$-connected rooted graph of order at least $\delta+3$ with minimum degree at least $\delta$. By \Cref{L2.2}, $G[X]$ consists of some components of order $\delta-1.$ By $\delta(G)\ge\delta,$ we have $G=G_2$ where $l_1\ge 2$, $l_2=1$, $|S|=0$ and $|L(\delta)|=0.$
	
		Thus, $G=G_2$, where $l_1\ge 1,l_2=1$, $|S|=0$ and $|L(\delta)|=0.$

		\item[$\bullet$]~ $|A_1|\geq |A_3|.$
		
		Now $\omega \ge \delta+2,$ there exists no $(v_s,v_t)$-path of order more than $\delta+2$ in $G[A_2\cup X]$.
		Recall that $|X|\ge \delta-1.$
		If $|X|=\delta-1,$ by the maximality of $G$, $A_2\cup X$ is a $(\delta+1)$-clique. Now 
		$G=G_2$ where where $l_1=1,l_2=1$, $|S|=0$ and $|L(\delta)|=0.$
		If $|X|=\delta,$ by the maximality of $G,$ $G=G_2$ where $l_1=0$, $l_2=2$, $|S|=0$ and $|L(\delta)|=0$.
		Suppose $|X|\ge \delta+1.$ Now $(G[A_2\cup X],v_s,v_t)$ is a $2$-connected rooted graph of order at least $\delta+3$ with minimum degree at least $\delta$.  By \Cref{L2.2}, $G[X]=L(\delta)\cup T(\delta)$; or $\delta=3$ and $G[X]=S\cup L(3)\cup T(3)$.
		If the former happens, by the maximality of $G,$ now $G=G_2$ where $l_1+|L(\delta)|\ge 1$, $l_2\ge 1$ and $|S|=0$.
		If the latter happens, $\delta=3,$ $G=G_2,$ where $l_1+l_2\ge 1$, $l_2\ge1,$ $|S|>0$ and $|L(\delta)|\ge 0$.
		
		Thus, $G=G_2$  where $l_2\ge 1$ and $X\cup A_3$ consists of at least two components.
		\vskip 3mm
		{\bf Case 2.} $(d_P(v_1),d_P(v_m))=(\omega,\delta)$.
		\vskip3mm 
		Since $N_H(v_1)\subseteq N_P(v_1)$ and $N_G(v_m)=N_P(v_m),$ we have $\omega(G)\ge \omega$ and $\delta(G)=\delta$. 
		By \Cref{cc5} -- \Cref{c6'}, it suffices to consider the case where $\delta\ge 3$ and neither $v_1$ nor $v_m$ is adjacent to any two consecutive vertices of $V(v_sPv_t).$ Recall that $X\ne\emptyset.$
		
		{\bf Subcase 2.1.}
		$(s,t)$ is not a minimal crossing pair.
		
		By \Cref{c8'}, $X$ is an independent set of $G$ and each vertex of $X$ can only be connected to $A_2\subseteq A_1\cup A_2\cup A_3$. By \Cref{c3}, $A_2=A_2'=\{v_s,v_{s+2},\cdots, v_{t-2},v_t\}.$
		By $\delta(G)\ge \delta$, we have $\left|A_2\right|\ge \delta$.
		Hence, $\left|A_2\right|= \delta$ and $\left|A_3\right|=1$ since $|A_2\cup A_3|=\delta+1$.
		
		Now, $G$ is the graph with a vertex partition $A_1\cup A_2\cup A_3\cup X$ with sizes $\omega-\delta+1$, $\delta$, $1$, and $n-\omega-2$ respectively, whose edge set consists of $[X\cup A_3, A_2]$, and edges in $A_1\cup A_2$.
		Since $G$ is non-hamiltonian, we have $n-\omega-2+1\ge \delta$, i.e., $n\ge \omega+\delta+1$.
		By the maximality of $G$, $G=H_1(n,\omega+1,\delta)$.
		
		{\bf Subcase 2.2.}
		$(s,t)$ is a minimal crossing pair.
		
		By \Cref{c3} and \Cref{c8'}, $N_G(A_1)\subseteq A_1\cup A_2$ and $N_G(A_3)\subseteq A_2\cup A_3$. Since $G$ is non-hamiltonian, we have $X\ne \emptyset.$ Then
		$d_X(z)\ge \delta-2$ for any vertex $z\in X$.
		Hence, $|X|\geq \delta-1$.
		Note that $|A_2|=2$, $|A_1\cup A_2|=\omega+1$ and $|A_2\cup A_3|=\delta+1$.
		Hence, $|A_1|=\omega-1\ge \delta-1=|A_3|$.
		By $c(G)=k-1$, $|V(v_{s+1}Pv_{t-1})|\le \min\{|A_1|,|A_3|\}=\delta-1$ and hence there exists no $(v_s,v_t)$-path of order more than $\delta+1$ in $G[A_2\cup X]$.

		If $\left|X\right|=\delta-1$, by $\delta(G)\ge\delta$, then $G[A_2\cup X]=K_{\delta+1}$ and then $G= H_2(n,\omega+1,\delta)$ where $l= 2.$
		If $|X|=\delta$, by $\delta(G)\ge\delta$ and \Cref{L2.2}, then $\delta=2$ and $X=\overline{K}_2$. Hence, $G=H_2(n,\omega+1,2)$ where $l= 3$.
		Suppose $|X|\ge \delta+1.$
		We consider a $2$-connected rooted graph $(G[A_2\cup X], v_s,v_t)$ of order at least $\delta+3$ with minimum degree at least $\delta$. By \Cref{L2.2}, $G[X]$ consists of some components of order $\delta-1$. By $\delta(G)\ge \delta,$ each component is a $(\delta-1)$-clique.
		By the maximality of $G$, $G[A_1\cup A_2]=K_{\omega+1}$ and $G[A_2\cup A_3]=K_{\delta+1}$.
		Then $G=H_2(n,\omega+1,\delta)$ where $l\ge 3.$

		Thus, $G=H_2(n,\omega+1,\delta)$ where $l\ge 2.$ 
		This  proves \Cref{l2.3}.
	\end{proof}
	
	We now prove our main result using an approach analogous to that in the proof of the \Cref{l2.3}.
	
	\vskip 3mm
	
	\begin{proof}[\bf  Proof of Theorem \ref{thm1.3}]
		If $G$ is hamiltonian, then $c(G)=n$ and we are done, so assume that $G$ is non-hamiltonian.
		To the contrary, suppose $c(G)<\delta+\omega+1$.
		Since $G$ is $2$-connected, by Dirac's theorem, $\omega\ge \delta\ge 2$.
		Let $k=\delta+\omega+1.$
		If $c(G)\leq k-2$, then $G=H_1(n,\omega,\delta)$ or $G=H_2(n,\omega,\delta)$ by {\blue Theorem} \ref{Yuan}.
		Hence, assume that $c(G)=k-1$.
		
		Suppose that $G$ is an edge-maximal graph, that is, $G+xy$ contains a cycle of length at least $k$ for any edge $xy\notin E(G).$ 
		Note that $\delta(G)\ge \delta$ and $\omega(G)\ge \omega.$ 
		Let $\omega(G)=\omega'$ and $\delta(G)=\delta'$. 
		Let $H= K_{\omega'}\subseteq G$ and $T=G[V(G)\setminus V(H)].$ 
		Clearly, there exist two vertices $x\in V(H)$ and $y\in V(T)$ such that $xy\notin E(G).$ Since $G+xy$ contains a cycle of length at least $k$, there is an $(x,y)$-path of length at least $k.$
		
		Define $\mathcal{P}$ as the set of all longest $(H,T)$-paths with nonadjacent endpoints.
		By \Cref{L1}, we need only consider the three types of $\mathcal{P}$:
		(1) $\mathcal{P}$ contains a path such that the degrees of its two endpoints are $\omega-1$ and $\delta+1$ in the path; (2) $\mathcal{P}$ does not contain any path described in (1), but there exists a path in $\mathcal{P}$ whose endpoints have degrees $\omega$ and $\delta$ in the path, respectively; (3) each path in $\mathcal{P}$ has endpoints with degrees $\omega-1$ and $\delta$ in the path.
		
		Note that $G$ is a $2$-connected edge-maximal non-hamiltonian graph with $\delta'\ge \delta$ and $\omega'\ge \omega$, where $\omega\ge \delta\ge2$.  Let $P=v_1v_2\dots v_m \in \mathcal{P}$. Now $m\ge k$. Since $P$ is a longest $(H,T)$-path, we have $V(H)\subseteq V(P).$ Thus, $d_P(v_1)\ge \omega'-1\ge \omega-1$.
		If $(d_P(v_1),d_P(v_m))=(\omega-1,\delta+1)~\text{or}~(\omega,\delta)$, 
		by \Cref{l2.3}, then $G\in\{H_1(n,\omega,\delta+1),H_1(n,\omega+1,\delta),H_2(n,\omega+1,\delta),H_3(n,\omega,2),G_1,G_2,G_3,G_4\}$. Subsequently, our analysis need only focus on $\mathcal{P}$ under the third type.
		Since $N_H(v_1)\subseteq N_P(v_1)$ and $N_G(v_m)=N_P(v_m),$ we have $\omega'=\omega$ and $\delta'=\delta.$
		
		We consider the following two cases.
		
		\begin{case}
			Each path in $\mathcal{P}$ has no crossing pair.
		\end{case}
		
		Now, $P$ has no crossing pair.
		By \Cref{L1} and $c(G)=k-1$, $N_P(v_1)\cap N_P(v_m)\neq \emptyset$ and $|N_P(v_1)\cap N_P(v_m)|=1.$
		Let $N_P(v_1)\cap N_P(v_m)=\{v_\alpha\}$ for some $2\le\alpha\le m-1.$
		Since $G$ is 2-connected, there exits a  $(v_1Pv_{\alpha-1},v_{\alpha+1}Pv_m)$-path $Q$ such that $V(Q)\cap V(P)=\{v_a,v_b\}$ with $2\leq a<\alpha<b\leq m-1.$
		Let $s'=\rm min$$\{h:h>a,v_h\in N_P(v_1)\}$ and let
		$t'=\rm max$$\{h:h<b,v_h\in N_P(v_m)\}.$
		Since there is a cycle $C_0=v_1Pv_aQv_bPv_mv_{t'}Pv_{s'}v_1,$
		we have $$k-1=c(G)\geq |C_0|\geq |N_P[v_1]|+|N_P[v_m]|-1+|Q|-2\geq k-1.$$
		This implies that $Q$ is an edge $v_av_b$, $|C_0|=k-1$ and $V(C_0)=N_P[v_1]\cup N_P[v_m]$. Recall that $P$ has no crossing pair and $N_P(v_1)\cap N_P(v_m)=\{v_\alpha\}$. Then
		\begin{align}\label{a2}
			~N_P[v_1]=V(v_1Pv_a)\cup V(v_{s'}Pv_\alpha) ~\text{and} ~N_P[v_m]=V(v_\alpha Pv_{t'})\cup V(v_bPv_m).
		\end{align}
		
		\begin{claim}\label{c6}
			$N_P[v_\beta]=N_P[v_1]$ for $1\leq \beta\leq a-1$ and $N_P[v_\gamma]=N_P[v_m]$ for $b+1\leq \gamma\leq m.$
		\end{claim}
		
		\begin{proof}
			For $1\leq \beta\leq a-1$, by (\ref{a2}), $v_1v_{\beta+1}\in E(G)$. Since $N_P[v_1]=V(H)$, we have $v_\beta\in V(H)$ and hence $N_P[v_1]\subseteq N_P[v_\beta]$. Next, we show that  $N_P[v_\beta]\subseteq N_P[v_1]$. Suppose $y\in N_P[v_\beta]\setminus N_P[v_1]$.
			Consider the path $P_\beta=v_\beta Pv_1v_{\beta+1}Pv_m$ for any $\beta$ with $1\leq \beta\leq a-1$.
			Clearly, $P_\beta \in \mathcal P$ and $P_\beta$ has no crossing pair. By (\ref{a2}), $y\in V(v_{a+1}Pv_{s'-1})$. Then $C_\beta=v_\beta P_\beta v_av_bP_\beta v_mv_{t'}P_\beta yv_\beta$ is a cycle such that $|C_\beta|\geq|C_0|+1=k$, a contradiction.
			Thus, $N_P[v_\beta]=N_P[v_1]$ for $2\leq \beta \leq a-1.$
			
			For $b+1\leq \gamma\leq m,$ by (\ref{a2}), $v_{\gamma-1}v_m\in E(G)$. Consider the path $P_\gamma=v_1Pv_{\gamma-1}v_mPv_\gamma$. Clearly, $P_\gamma \in \mathcal P$ and $P_\gamma$ has no crossing pair. Then $d_P(v_\gamma)=d_{P_\gamma}(v_\gamma)=d_G(v_\gamma)=\delta.$ Suppose $z\in N_P[v_\gamma]\setminus N_P[v_m]$. By (\ref{a2}), $z\in V(v_{t'+1}Pv_{b-1})$ and then $C_\gamma=v_1P_\gamma v_av_bP_\gamma v_\gamma zP_\gamma v_s'v_1$ is a cycle such that $|C_\gamma|\geq|C_0|+1=k$, a contradiction. Thus, $N_P[v_\gamma]\subseteq N_P[v_m]$ and hence $N_P[v_\gamma]=N_P[v_m]$ for $b+1\leq \gamma\leq m.$ This proves \Cref{c6}.
		\end{proof}
		
		By the maximality of $P$ and the proof of \Cref{c6}, we have $N_P[v_1]=V(H)$ and $N_G[v_\gamma]=N_P[v_\gamma]$ for $b+1\le \gamma\le m. $
		
		\begin{claim}\label{c7}
			$a<\alpha-1<s'=t'=\alpha<\alpha+1< b$.
		\end{claim}
		
		\begin{proof}
			To the contrary, suppose $s'<\alpha.$
			By (\ref{a2}), $v_{\alpha-1}\in N_P(v_1).$
			By \Cref{c6}, $v_{a-1}v_\alpha\in E(G)$.
			There is a path $L=v_aPv_{\alpha-1}v_1Pv_{a-1}v_\alpha Pv_m \in \mathcal P$.
			Since $v_av_b,v_\alpha v_m\in E(G)$, $L$ has a crossing pair, a contradiction.
			Thus, $s'=\alpha.$ Similarly, $t'=\alpha.$
			
			To the contrary, suppose $a=\alpha-1$, i.e., $v_av_\alpha\in E(P).$
			By $v_{\alpha}\in V(H) $,  there is a path $M=v_aPv_1v_\alpha Pv_m\in \mathcal P.$
			Since $v_av_b, v_\alpha v_m\in E(G)$, $M$ has a crossing pair, a contradiction.
			Thus, $a<\alpha-1$. Similarly, $\alpha+1<b.$
			This proves \Cref{c7}.
		\end{proof}
		
		We define the following sets.
		\begin{align}
			A_1&=V(v_1Pv_{a-1}), B_1=V(v_{b+1}Pv_m), C_1=\{v_a,v_\alpha,v_b\}, \notag\\
			D_1&=V(v_{a+1}Pv_{\alpha-1}),D_2=V(v_{\alpha+1}Pv_{b-1}),X=V(G)\setminus V(P),\notag\\
			X_1&=\{x\in X~|~x ~\text{is only connected to} ~D_1\cup\{v_a,v_\alpha\}\subseteq V(P)\} ~\text{and} \notag\\
			X_2&=\{x\in X~|~x ~\text{is only connected to} ~D_2\cup \{v_\alpha,v_b\}\subseteq V(P)\}.\notag
		\end{align}
		Clearly, $X_1\ne\emptyset$, $X_2\ne\emptyset$ and $X_1\cap X_2=\emptyset$. By \Cref{c7}, $D_1\neq \emptyset$ and $D_2\neq \emptyset$.
		Note that $N_P[D_1]\subseteq D_1\cup \{v_a,v_\alpha\}$ and $N_P[D_2]\subseteq D_2\cup \{v_\alpha,v_b\}$.
		Recall that $C_0=v_1Pv_av_bPv_mv_\alpha v_1$ with $|C_0|=k-1.$
		Since $G$ is 2-connected, if $[X_1,X_2]\ne\emptyset$, then there exists a cycle of length at least $k$ containing all vertices of $C_0$ and some vertices of $X_1\cup X_2$, a contradiction. Thus, $[X_1,X_2]=\emptyset$.
		
		\begin{claim}\label{c8}
			$X=X_1\cup X_2$.
		\end{claim}
		
		\begin{proof}
			To the contrary, suppose $x\in X\setminus (X_1\cup X_2)$.
			Since $G$ is 2-connected, there are two vertex-disjoint $(x,P)$-paths $P_1$ and $P_2$ with $P_j\cap V(P)=\{y_i\}$ for $i=1,2$. Assume that $y_1\in V(v_1Py_2).$ By \Cref{c6}, $N_G(B_1)\subseteq B_1\cup C_1$ and hence $\{y_1,y_2\}\cap B_1=\emptyset.$
			Since $A_1\cup \{v_a,v_\alpha\}$ is a clique, we have $y_1\notin A_1$ or $y_2\notin A_1$.
			Otherwise, there is an $(H,T)$-path longer than $P,$ a contradiction. Suppose $y_1\in A_1$ and $y_2\notin A_1$. Without loss of generality, let $y_1=v_1$. Other situations are similar.
			If $y_2\in D_1\cup  \{v_a\}$, by \Cref{c6}, then $v_\alpha v_{a-1}\in E(G)$ and then $v_1P_1xP_2y_2Pv_av_bPv_mv_\alpha v_{a-1}Pv_1$ is a cycle of length at least $ k.$  If $y_2\in D_2\cup  \{v_b\}$, then $v_1P_1xP_2y_2Pv_mv_\alpha Pv_1$ is a cycle of length at least $k$.
			If $y_2=v_\alpha$, then $v_1Pv_av_bPv_mv_\alpha P_2xP_1 v_1$ is a $k$-cycle. Both situations lead to a contradiction.
			
			Hence, $y_1,y_2\notin A_1$. By $x\in X\setminus (X_1\cup X_2)$ and $y_1\in V(v_1Py_2)$, we have $y_1\in D_1\cup \{v_a\}$ and $y_2\in D_2\cup \{v_b\}.$
			Now $v_1Py_1P_1xP_2y_2Pv_mv_\alpha v_1$ is a cycle of length at least $ k$, a contradiction. Thus, $X=X_1\cup X_2$. This proves \Cref{c8}.
		\end{proof}
		
		By $c(G)=k-1,$ we have $|D_1|\le \min\{|A_1|,|B_1|+1\}=\min\{\omega-2,\delta\}$, that is,
		\begin{align}\label{d1}
			\text{there exists no} ~(v_a,v_\alpha)\text{-path of order more than} \min\{\omega-2,\delta\}+2.
		\end{align}
		Similarly, we have $ |D_2|\le \min\{|A_1|+1,|B_1|\}=\min\{\omega-1,\delta-1\}=\delta-1,$ that is,
		\begin{align}\label{d2}
			\text{there exists no} ~(v_\alpha,v_b)\text{-path of order more than} ~\delta+1.
		\end{align}

		Note that both $(G[X_1\cup D_1\cup\{v_a,v_\alpha\}],v_a,v_\alpha)$ and $(G[X_2\cup D_2\cup \{v_\alpha,v_b\}],v_\alpha,v_b)$ are $2$-connected rooted graphs with minimum degree at least $\delta.$ By $\delta(G)=\delta$, $|X_i\cup D_i|\geq \delta-1$ for $i=1,2.$
		
		First, we assert that $G[X_2\cup D_2]$ consists of some $(\delta-1)$-cliques.
		If $|X_2\cup D_2|=\delta-1$, then $X_2\cup D_2$ is a $(\delta-1)$-clique.
		Assume that $|X_2\cup D_2|\ge \delta.$
		If $|X_2\cup D_2|=\delta$, it is easy to verify that there is a $(v_\alpha,v_b)$-path of order more than $\delta+1$, which contradicts (\ref{d2}).
		Hence, $|X_2\cup D_2|\ge \delta+1$.
		By \Cref{L2.2} and (\ref{d2}), $G[X_2\cup D_2]$ consists of some components of order $\delta-1$, and each component is $(\delta-1)$-clique since $\delta(G)=\delta.$  Thus, $G[X_2\cup D_2]$ consists of some $(\delta-1)$-cliques.
		
		Next, we analyzes the structure of $G[X_1\cup D_1]$. Recall that $\omega'=\omega$ and $\delta'=\delta$.
		We assert that 
		$\omega>\delta.$ Otherwise, suppose $\omega=\delta.$ By \Cref{c6} and \Cref{c8},  $A_1=\emptyset$. Then $\omega=\delta=2$. Now $\{v_a,v_\alpha,v_b\}$ is a $3$-clique, contradicting $\omega'=\omega=2$.
		Hence, $\omega\ge \delta+1$.
		
		Suppose $\omega= \delta+1$.
		If $|X_1\cup D_1|=\delta-1$, then $X_1\cup D_1$ is a $(\delta-1)$-clique.  If $|X_2\cup D_2|=\delta$, it is easy to verify that there is a $(v_a,v_\alpha)$-path of order more than $\delta+1$, which contradicts (\ref{d2}).
		If $|X_1\cup D_1|\ge \delta+1$, by \Cref{L2.2} and (\ref{d1}), $G[X_1\cup D_1]$ consists of some components of order $\delta-1$ and each component is $(\delta-1)$-clique since $\delta(G)=\delta.$
		
		Suppose $\omega\ge\delta+2.$
		If $|X_1\cup D_1| \le \delta,$ by the maximality of $G$, then $X_1\cup D_1$ is a $(\delta-1)$-clique or $\delta$-clique.
		If $|X_1\cup D_1|\ge \delta+1$, by \Cref{L2.2}, then $G[X_1\cup D_1]=L(\delta)\cup T(\delta)$ or $\delta =3$ and $G[X_1\cup D_1]=S\cup L(3)\cup T(3).$ 
		
		By the maximality of $G$, if $\omega=\delta+1,$ then $G=G_3$ where $l_1\ge1$, $l_3\ge 2$ and $l_2=|S|=|L(\delta)|=0$; and if $\omega\ge\delta+2,$ then $G=G_3$ where $l_1+l_2\ge1$, $|S|\ge0$, $|L(\delta)|\ge 0$ and $l_3\ge 2.$

		\begin{case}
			There exists a path in $\mathcal P$ having a crossing pair.
		\end{case}

		Let $P=v_1v_2\dots v_m$ be a path in $\mathcal{P}$ that has a crossing pair where $m\ge k$.
		Let $s=\min\{i:v_iv_m\in E(G)\}$ and $t=\max\{j:v_jv_1\in E(G)\}.$ Clearly, $s\ge 2$, $t\le m-1$ and $s\le t-2.$
		Let $(p,q)$ be a minimum crossing pair. There is a cycle $\widetilde{C}=v_1Pv_pv_mPv_qv_1$.
		Note that $N_P^-(v_1)\cap N_P[v_m]=\emptyset$ and $N_P[v_1]\cap N_P^+(v_m)=\emptyset$. We have
		$$k-2= |(N_P^-(v_1)\cup N_P[v_m])\setminus \{v_{q-1}\}|\leq |\widetilde{C}|\leq c(G)= k-1.$$
		We consider the following subcases.
		
		\vskip 3mm
		\textbf{Subcase~2.1.} $|\widetilde{C}|=k-2.$
		\vskip 3mm
		Since $N_P^-(v_1)\cap N_P[v_m]=\emptyset$ and $N_P[v_1]\cap N_P^+(v_m)=\emptyset$, we have
		\begin{align}\label{a3}
			V(\widetilde{C})=(N_P^-(v_1)\cup N_P[v_m])\setminus \{v_{q-1}\}=(N_P[v_1]\cup N_P^+(v_m))\setminus \{v_{p+1}\}.
		\end{align}
		Similarly, we have
		\begin{align}\label{f1}
			N_P[v_i]=N_P[v_1] ~\text{for} ~1\leq i\leq s-1 ~\text{and}~N_P[v_j]=N_P[v_m] ~\text{for}~ t+1\leq j\leq m.
		\end{align}

		\begin{claim}\label{c10}
			$(s,t)=(p,q)$.
		\end{claim}
		
		\begin{proof}
			To the contrary, suppose $(s,t)\neq (p,q)$.  Without loss of generality, let $q<t.$ Then $v_{t-1},v_t\in V(\widetilde{C})$. We show that $v_m$ is not adjacent to any two consecutive vertices of $V(v_sPv_t)$.
			Otherwise, suppose  $\{v_f,v_{f+1}\}\subseteq N_P(v_m)$. By (\ref{f1}), $v_{t+1}v_{f+1}\in E(G)$ and hence $v_1Pv_fv_mPv_{t+1}v_{f+1}Pv_tv_1$ is an $m$-cycle, a contradiction. Similarly, $v_1$ is not adjacent to any two consecutive vertices of $V(v_sPv_t)$.
			
			By $v_1v_t,v_mv_t\in E(G)$, $v_{t-1}\notin N_P(v_1)\cup N_P(v_m)$. Hence, $q\le t-2,$ that is, $v_{t-2}\in V(\widetilde{C}).$
			By (\ref{a3}), $v_{t-1}\in N_P^+(v_m)$ and hence $v_{t-2}\in N_P(v_m).$ Now $(t-2,t)$ is a minimal crossing pair.  $v_1Pv_{t-2}v_mPv_tv_1$ is a cycle of length $m-1>k-2$, a contradiction.
			This proves \Cref{c10}.
		\end{proof}

		We define the following sets.
		\begin{align}
			A_1&=V(v_1Pv_{s-1}), A_2=\{v_s,v_t\}, A_3=V(v_{t+1}Pv_m) ~\text{and}~X=V(G)\setminus A_1\cup A_2\cup A_3.\notag
		\end{align}
		By (\ref{f1}) and \Cref{c10}, $G[A_1\cup A_2]=K_\omega$, $G[A_2\cup A_3]=K_{\delta+1}$. Similarly, we have
		\begin{align}\label{f2}
			~\text{each vertex of}~ X ~\text{is only connected to}~ A_2\subseteq A_1\cup A_2\cup A_3.
		\end{align}

		By $c(G)=k-1,$  $|X|\ge\delta$.
		If $|X|=\delta$, then $G[X]$ is a component of order $\delta.$  By the maximality of $G$, $X\cup A_2$ is a $(\delta+2)$-clique. Let $u_1P_1u_{\delta}$ be a Hamilton path in $G[X]$ such that $u_1v_s, u_\delta v_t\in E(G)$. Now there is a path $P'=v_1Pv_sv_mPv_tu_\delta P_1u_1 \in \mathcal P$ with $d_{P'}(v_1)=\omega-1$ and $d_{P'}(u_1)=\delta+1$ (type (1)), a contradiction.

		Hence, $|X|\ge \delta+1.$ By \Cref{c10} and (\ref{f2}), $(G[\{v_s,v_t\}\cup X],v_s,v_t)$ is a 2-connected rooted graph of order at least $\delta+3$ with minimum degree at least $\delta$.
		By \Cref{L2.2} and $c(G)=k-1$, $X=L(\delta)\cup T(\delta)$ or $\delta=3$ and $X=S\cup L(3)\cup T(3)$.
		
		If the former case holds, by $c(G)=k-1$, either $L(\delta)\ne \emptyset$ or $T(\delta)$ contains a component of order $\delta$. Similar to the above argument, there is always an $(H,T)$-path of type (1).        
		Hence, the latter case holds. Clearly, $S$ contains exactly one star, $L(\delta)=\emptyset$ and $T(\delta)$ consists of components of $\delta-1$.
		
		Thus, $G=G_2$ where $l_1\ge1$, $l_2=0$, $|L(\delta)|=0$ and $|S|> 0$.

		\vskip 3mm
		
		\textbf{Subcase~2.2.} $|\widetilde{C}|=k-1$.
		\vskip 3mm
		Now $|V(\widetilde{C})\setminus (N_P^-(v_1)\cup N_P[v_m])|=1$, similarly, $|V(\widetilde{C})\setminus (N_P[v_1]\cup N_P^+(v_m))|=1$.
		Let $v_h\in V(\widetilde{C})\setminus (N_P^-(v_1)\cup N_P[v_m])$. Then,
		\begin{align}\label{a4}
			\text{each vertex in} ~V(\widetilde{C})\setminus \{v_h\} ~\text{belongs to} ~N_P^-(v_1) ~\text{or} ~N_P[v_m].
		\end{align}
		Let $v_{h'}\in V(\widetilde{C})\setminus (N_P[v_1]\cup N_P^+(v_m))$. Then,
		\begin{align}\label{a5}
			\text{each vertex in} ~V(\widetilde{C})\setminus \{v_{h'}\} ~\text{belongs to} ~N_P[v_1] ~\text{or} ~N_P^+(v_m).
		\end{align}
		Recall that $(p,q)$ is a minimum crossing pair and $\widetilde{C}=v_1Pv_pv_mPv_qv_1$.
		Clearly, $h\neq p$ and $h'\neq q$. We assert that $h'=h+1$.
		Otherwise, suppose $h'\neq h+1$. By $v_h\notin N_P^-(v_1)$, $v_{h+1}\notin N_P[v_1]$. Since $h\ne p$, $v_{h+1}\in V(\widetilde{C})$. By (\ref{a5}) and $h'\ne h+1$, $v_{h+1}\in N_P^+(v_m)$ and then $v_h\in N_P(v_m)$, a contradiction.
		Hence, $h'=h+1$.
		By $v_h\notin N_P^-(v_1)\cup N_P[v_m]$, $h\notin \{1,s\}$, similarly, $h'\notin \{t,m\}$, i.e., $h\notin \{t-1,m-1\}$. Thus, $h\notin \{1,s,p,t-1,m-1\}$.
		
		We need to consider the following three cases.
		
		\vskip 3mm
		{\bf Case A.} $2\le h<h+1\le s$.
		\vskip 3mm

		We assert that $h=s-2$. Otherwise, suppose $h\le s-3$. By (\ref{a4}), $V(v_1Pv_s)\setminus \{v_{h+1}\}\subseteq N_P[v_1]$. Hence, $v_{h+3}v_1\in E(G)$ and then there is a path $P'=v_{h+2}Pv_1v_{h+3}Pv_m\in \mathcal{P}$ with $d_{P'}(v_{h+2})\ge \omega$ and $d_{P'}(v_m)=\delta,$ a contradiction. Suppose $h=s-1$. By $v_h\in V(H),$ there is a path $P''=v_hPv_1v_tPv_sv_mPv_{t+1}\in\mathcal{P}$ with $d_{P''}(v_{h})\ge \omega$ and $d_{P''}(v_{t+1})=\delta,$ a contradiction.
		Hence, $h=s-2$ which implies that $s\ge 4.$
		
		\begin{claim}\label{c12}
			$(\rm i)$ For any vertex  $v_i\in V(v_1Pv_{s-3}),$ $N_P[v_i]=N_P[v_1]$. Moreover, $v_1$ is not adjacent to any two consecutive vertices of $V(v_sPv_t).$
			
			$(\rm ii)$ For any vertex $v_j\in  V(v_{t+1}Pv_{m})$, $N_P[v_j]=N_P[v_m]$. Moreover, $v_m$ is not adjacent to any two consecutive vertices of $V(v_sPv_t).$
		\end{claim}
		
		\begin{proof}
			We first show that $N_P[v_i]=N_P[v_1]$ for $v_i\in V(v_1Pv_{s-3})$.
			By $d_P(v_1)=\omega-1$ and $V(H)\subseteq V(P)$, $N_P[v_1]=V(H).$
			Let $v_i\in V(v_1Pv_{s-3})$. By (\ref{a5}), $N_P[v_{i}]\subseteq N_P[v_1]$ and then there is a path $v_iPv_1v_{i+1}Pv_m\in \mathcal{P}$.
			By the choice of $P$, $d_{P}(v_i)= \omega-1$ and $N_P[v_i]=N_P[v_1]$.
			Suppose $v_i,v_{i+1}\in V(v_sPv_t)\cap N_P(v_1)$. By $v_{h-1}\in V(v_1Pv_{s-3}),$ $v_iv_{h-1}\in E(G)$ and there is a path
			$P'=v_{h}Pv_iv_{h-1}Pv_1v_{i+1}Pv_m \in \mathcal P$ with $d_{P'}(v_h)\ge\omega$ and $d_{P'}(v_m)=\delta$, a contradiction.
			
			The proof of $\rm (ii)$ is similar, so we omit the proof here.
			This proves \Cref{c12}.
		\end{proof}

		Let $A_1=V(v_1Pv_{s-2})$, $A_2=N_P(v_1)\cap V(v_sPv_t)$, $A_2'=N_P(v_m)\cap V(v_sPv_t)$,
		$A_3=V(v_{t+1}Pv_m)$ and let $X=V(G)\setminus (A_1\cup A_2\cup A_3)$. Clearly, $X\ne\emptyset.$
		
		If $(s,t)$ is a minimal crossing pair, then $A_2=A_2'=\{v_s,v_t\}$. Assume that $(s,t)$ is not a minimal crossing pair. By (\ref{a4}), (\ref{a5}) and \Cref{c12}, we have $A_2=A_2'=\{v_s,v_{s+2},...,v_{t-2},v_t\}$ and
		$t\equiv s\pmod 2$.

		Clearly, $v_{s-1}=v_{h+1}\in X$ and we have $G[A_1\cup A_2]=K_\omega$ and $G[A_2\cup A_3]=K_{\delta+1}$. Note that a minimal crossing pair is also a minimum crossing pair. Moreover, if $(s,t)$ is not a minimal crossing pair, then $m=k.$

		\begin{claim}\label{c16}
			Each vertex of $X$ is only connected to $A_2\cup\{v_h\}\subseteq A_1\cup A_2\cup A_3$. Moreover, if $|A_2|\geq 3$, then $X$ is an independent set of $G$.
		\end{claim}
		
		\begin{proof}
			By \Cref{c12}(ii) and the choice of $P$, for $j\in [t+1,m]$, $N_G(v_j)=N_P(v_j)=N_G(v_m)=N_P(v_m)\subseteq A_2\cup A_3,$ i.e., $[X,A_3]=\emptyset$.
			By $N_P[v_1]\subseteq A_1\cup A_2$ and \Cref{c12}, we have $[X\cap V(P),A_1\setminus \{v_h\}]=\emptyset$.
			Thus, each vertex of $X\cap V(P)$ is only connected to $A_2\cup\{v_h\}\subseteq A_1\cup A_2\cup A_3$.
			
			Suppose that there is a vertex  $x \in X\setminus V(P)$ such that $x$ is not only connected to $A_2\cup\{v_h\}\subseteq A_1\cup A_2\cup A_3$.
			Since $G$ is 2-connected, there are two vertex-disjoint $(x,P)$-paths $P_1$ and $P_2$ with $V(P_i)\cap V(P)=\{x_i\}$ for $i=1,2$. Assume that $x_1\in V(v_1Px_2).$
			If $x_1,x_2\in A_1\setminus \{v_h\}$, then $v_1Px_1P_1xP_2x_2Px_1^+x_2^+Pv_m$ is an $(H,T)$-path longer than $P$, a contradiction. If $x_1\in A_1\setminus \{v_h\}$ and $x_2=v_{h+1},$ then  $v_hPx_1^+v_1Px_1P_1xP_2v_{h+1}Pv_m$ is an $(H,T)$-path longer than $P$, a contradiction. If $x_1\in A_1\setminus \{v_h\}$ and $x_2\in (X\cap V(P))\setminus \{v_{h+1}\},$ then there is a cycle of length at least $ k$, a contradiction. Suppose $x_1\in A_1\setminus \{v_h\}$ and $x_2\in A_2.$ Without loss of generality, let $|A_2|\ge 3$. The case of $|A_2|=2$ is similar. If $x_2\ne v_s$, then there is a cycle $v_1Px_1P_1xP_2x_2Pv_mx_2^{-2}Px_1^+v_1$ of length at least $ k$.  If  $x_2=v_s$, then  $P'=v_hv_{h+1}x_2P_2xP_1x_1Pv_1x_1^+Pv_h^-x_2^{+2}Pv_m$ when $x_1v_h\notin E(P)$ or $P'=v_hv_{h+1}x_2P_2xP_1x_1Pv_1x_2^{+2}Pv_m$ when $x_1v_h\in E(P)$ is an $(H,T)$-path with $d_{P'}(v_h)\ge\omega$ and $d_{P'}(v_m)=\delta$. Both situations lead to a contradiction. Thus, each vertex of $X$ is only connected to $A_2\cup\{v_h\}\subseteq A_1\cup A_2\cup A_3$.
			
			Let $|A_2|\ge 3$. Clearly, $m=k$.
			We assert that $[\{v_{h+1}\},X\cap V(P)]=\emptyset$. Otherwise, there is an $m$-cycle, a contradiction. Suppose that there is a vertex $x\in X\setminus V(P)$ such that $xv_{h+1}\in E(G)$. Since each vertex of $X$ is only connected to $A_2\cup\{v_h\}\subseteq A_1\cup A_2\cup A_3$, there is an $m$-cycle, a contradiction. Hence, $[\{v_{h+1}\},X]=\emptyset$.
			The rest of the proof is similar to that of \Cref{c7'}. Thus, $X$ is an independent set of $G$.
			This proves \Cref{c16}.
		\end{proof}

		Suppose $|A_2|\ge 3.$ By \Cref{c16}, $X$ is an independent set of $G$ and $|A_2|\ge \delta-1$.
		By \Cref{c12}(ii), $[\{v_h\}\cup X,A_3]=\emptyset$.
		Note that $G[A_2\cup A_3]=K_{\delta+1}$.
		Suppose $|A_2|= \delta-1$.
		Then $G[A_3]=K_2$ and $[X,A_2\cup\{v_h\}]$ is complete.
		Now $G+yv_h$ contains no cycle of length at least $ k$ with $y\in A_3$, contradicting to the maximality of $G.$
		Suppose $|A_2|= \delta$. Then $A_3=\{v_m\}$.
		By the maximality of $G$, $[X,A_2\cup\{v_h\}]$ is complete.
		Now $G+yv_h$ contains no cycle of length at least $ k$ with $y\in A_3$, contradicting to the maximality of $G.$
		
		Let $|A_2|=2$.
		By $c(G)=k-1$, $[\{v_{h+1}\}, X]=\emptyset$ and then $N_G(v_{h+1})\subseteq\{v_{h},v_s,v_t\}$. Now $\omega(G)=\omega>3$.
		Suppose $N_G(v_{h+1})=\{v_{h},v_s\}$. Then $\delta(G)=2$. Hence, $\delta=2$ and $|A_3|=1.$ Recall that each vertex of $X$ is only connected to $A_2\cup\{v_h\}\subseteq A_1\cup A_2\cup A_3$. By $c(G)=k-1$, $X\setminus\{v_{h+1}\}$ is an independent set of $G$.
		By $v_hv_m\notin E(G),$
		$G+v_mv_h$ contains no cycle of length at least $ k$, contradicting to the maximality of $G.$
		Thus,  $N_G(v_{h+1})=\{v_{h},v_s,v_t\}$. Then $|A_3|=1$ or $|A_3|=2$. By the above argument, we have $|A_3|=2$ and $\delta=3.$ Clearly, $X\setminus\{v_{h+1}\}$ consists of some $K_1$ or $K_2.$ Moreover, $v_h$ is not adjacent to any vertex of $K_2$ in $X$. 
		
		Thus, $G=H_4(n,\omega,3)$, where $l_1+l_2\ge 3$.

		\vskip 3mm
		{\bf Case B.} $t\le h<h+1\le m-1$.
		\vskip 3mm
		
		Recall that $s\ge 2$ and $t\le m-1.$ By (\ref{a4}) and (\ref{a5}), we have
		\begin{align}\label{a0}
			V(v_{1}Pv_{s})\subseteq N_P[v_1]~\text{and}~V(v_{t}Pv_{m})\setminus \{v_h\}\subseteq N_P[v_m].
		\end{align}
		
		\begin{claim}\label{c17}
			$N_P[v_i]=N_P[v_1]$ for $1\leq i\leq s-1$. Moreover, $v_1$ is not adjacent to any two consecutive vertices of $V(v_sPv_t).$
		\end{claim}
		\begin{proof}
			If $s=2$, we are done, so assume that $s\ge3$.
			By (\ref{a0}), $V(v_{2}Pv_{s})\subseteq V(H)$ and then $N_P[v_1]\subseteq N_P[v_i]$, for $2\leq i\leq s-1$.
			By the choice of $P$, $d_P(v_i)=d_P(v_1)=\omega-1.$
			Thus, $N_P[v_i]= N_P[v_1]$ for $1\leq i\leq s-1$.
			Suppose $v_{j-1},v_{j}\in N_P(v_1)\cap V(v_sPv_t)$.
			Since $N_P[v_i]= N_P[v_1]$ for $1\leq i\leq s-1$, $v_{s-1}v_j\in E(G)$.
			Then $v_1Pv_{s-1}v_jPv_mv_sPv_{j-1}v_1$ is an $m$-cycle, a contradiction.
			This proves \Cref{c17}.
		\end{proof}

		\begin{claim}\label{c18}
			
			$(\rm i)$ For any vertex $x$ of $ V(v_{t+1}Pv_{m})\setminus \{v_h,v_{h+1}\}$,  $N_P[x]\setminus N_P[v_m]=\emptyset$ or $\{v_h\}$.
			
			$(\rm ii)$ If $h\neq t$, then  $N_G(v_h)=N_P(v_h)=N_P(v_m)$.
			
			$(\rm iii)$ If $h=t$, then $N_P[v_{h+1}]\setminus\{v_h\}\subseteq N_P[v_m]$ and $|N_P[v_{h+1}]\cap N_P[v_m]|=\delta-1$.
			
		\end{claim}
		
		\begin{proof}
			Recall that $(p,q)$ is a minimum crossing pair and $\widetilde{C}=v_1Pv_pv_mPv_qv_1.$
			First, we prove (i). Let $x\in V(v_{t+1}Pv_{m})\setminus \{v_h,v_{h+1}\}$. We have $x^-\notin N_P^-(v_1)$ and $x^-\ne v_{h}$. By (\ref{a5}), $x^-v_m\in E(G)$.
			To the contrary, suppose $v_l\in N_P[x]\setminus (N_P[v_m]\cup \{v_h\})$.
			By $v_sv_m\in E(G)$, $l\ge s+1.$
			If $v_l\notin V(\widetilde{C})$, then  $v_1Pv_lxPv_mx^-Pv_qv_1$ is a cycle of length at least $k$, a contradiction.
			Hence, $v_l\in V(\widetilde{C})$. By $v_l\notin N_P[v_m]\cup\{v_h\}$ and (\ref{a4}), $v_l\in N^-_P(v_1)$ and hence $v_{l+1}\in N_P(v_1)$. Now $v_1Pv_lxPv_mx^-Pv_{l+1}v_1$ is an $m$-cycle, a contradiction. Thus, $N_P[x]\subseteq N_P[v_m]\cup \{v_h\}$.
			By $x^-v_m\in E(G)$, there is a path $v_1Px^-v_mPx\in \mathcal{P}$. By the choice of $P$,  $d_P(x)=d_G(x)=\delta.$ Thus, $N_P[x]\setminus N_P[v_m]=\emptyset~\text{or}~\{v_h\}.$
			
			Next we prove (ii). By $h\ne t$, we have $v_mv_{h-1}\in E(G)$. Consider the path $v_1Pv_{h-1}v_m\allowbreak Pv_h\in \mathcal{P}$. Then $N_G(v_h)=N_P(v_h).$ Suppose there is a vertex $y\in N_P(v_h)\setminus N_P(v_m)$.
			If $y\in V(\widetilde{C}),$ by (\ref{a4}), then $y\in N_P^-(v_1)$ and then $v_1y^+
			\in E(G)$.
			Now $v_1Pyv_hPv_mv_{h-1}Py^+v_1$ is a cycle of length $m\ge k,$ a contradiction.
			If $y\notin V(\widetilde{C}),$ then $v_1Pyv_hPv_mv_{h-1}Pv_qv_1$ is a cycle of length at least $k$, a contradiction.
			Hence,  $N_G(v_h)=N_P(v_h)=N_P(v_m)$.

			Now we prove (iii). Suppose to the contrary that  $v_l\in N_P(v_{h+1})\setminus (N_P[v_m]\cup \{v_h\})$.
			We assert that $v_l\notin V(v_1Pv_{s-1}).$ Otherwise, if $v_l\in V(v_1Pv_{s-2}),$ by \Cref{c17}, then $v_1Pv_lv_{h+1}Pv_mv_sP\\v_tv_{s-1}Pv_{l+1}v_1$ is an $m$-cycle.
			If $v_l=v_{s-1}$, by \Cref{c17}, then $v_1Pv_lv_{h+1}Pv_mv_sPv_tv_1$ is an $m$-cycle. Both situations lead to a contradiction.
			Thus, $v_l\in V(v_{s+1}Pv_{t-1})$.
			If $v_l\notin V(\widetilde{C})$, then $v_1Pv_pv_mPv_{h+1}v_lPv_tv_1$ is a cycle of length at least $k$, a contradiction.  If $v_l\in V(\widetilde{C})$, by $v_l\notin N_P[v_m]\cup \{v_h\}$ and (\ref{a4}), then $v_l\in N_P^-(v_1)$ and $v_{l+1}v_1\in E(G).$ By \Cref{c17}, $v_l\notin N_P[v_1]$.  By (\ref{a5}), $v_l\in N_P^+(v_m)$ and hence $v_{l-1}v_m\in E(G)$.
			Now $v_1Pv_{s-1}v_tPv_lv_{h+1}Pv_mv_{l-1}Pv_sv_1$ is an $m$-cycle, a contradiction.
			Thus, $N_P[v_{h+1}]\setminus\{v_h\}\subseteq N_P[v_m]$.
			Since $d_P(v_{h+1})\ge \delta$, $|N_P[v_{h+1}]\cap N_P[v_m]|=\delta-1$.
			This proves \Cref{c18}.
		\end{proof}

		\begin{claim}\label{c20}
			$v_m$ is not adjacent to any two consecutive vertices of $V(v_sPv_t).$
		\end{claim}
		
		\begin{proof}
			To the contrary, suppose $v_i,v_{i+1}\in N_P(v_m)\cap V(v_sPv_t)$.
			Suppose $h=t$.
			By \Cref{c18}(iii), $v_iv_{h+1}\in E(G)$ or $v_{i+1}v_{h+1}\in E(G)$.
			If $v_iv_{h+1}\in E(G)$, then $v_1Pv_iv_{h+1}Pv_mv_{i+1}Pv_tv_1$ is an $m$-cycle, a contradiction.
			If $v_{i+1}v_{h+1}\in E(G)$, then $v_1Pv_{s-1}v_tPv_{i+1}v_{h+1}Pv_mv_iPv_sv_1$ is an $m$-cycle, a contradiction.
			Hence, $h\neq t$. Note that $v_{m-1}v_{t+1}, v_{m-1}v_h, v_mv_{h+1}\in E(G).$
			By \Cref{c18}(ii), $v_{i+1}v_h\in E(G)$. Now $v_1Pv_iv_mv_{h+1}Pv_{m-1}v_{t+1}Pv_hv_{i+1}Pv_tv_1$ is an $m$-cycle, a contradiction. This proves \Cref{c20}.
		\end{proof}
		
		\begin{claim}\label{c20'}
			If $h\neq t$, then $N_P[v_{h+1}]\setminus\{v_h\}\subseteq N_P[v_m]$ and $|N_P[v_{h+1}]\cap N_P[v_m]|=\delta-1$.
		\end{claim}
		
		\begin{proof}
			To the contrary, suppose there is a vertex $v_l\in N_P(v_{h+1})\setminus (N_P[v_m]\cup \{v_h\})$. First, we show that $V(v_{t+1}Pv_m)=\{v_h,v_{h+1},v_m\}$.
			Suppose $V(v_{t+1}Pv_m)\neq \{v_h,v_{h+1},v_m\}$. Hence, $m\ne h+2$ or $h\neq t+1$. Suppose $m\neq h+2$. By \Cref{c18}(ii), $N_P(v_h)=N_P(v_m)$ and then $v_hv_{h+2}\in E(G)$. If $v_l\in V(\widetilde{C}),$ by $v_l\notin N_P[v_m]\cup \{v_h\}$ and (\ref{a4}), then $v_l\in N_P^-(v_1)$ and hence $v_{l+1}v_1\in E(G)$. Now $v_1Pv_lv_{h+1}v_hv_{h+2}Pv_mv_{h-1}Pv_{l+1}v_1$ is an $m$-cycle, a contradiction. If $v_l\notin V(\widetilde{C}),$ then $v_1Pv_l\allowbreak v_{h+1}
			v_hv_{h+2}Pv_mv_{h-1}Pv_qv_1$ is a cycle of length at least $k$, a contradiction.
			Suppose $h\neq t+1.$
			If $v_l\notin V(\widetilde{C})$, then $v_1Pv_lv_{h+1}Pv_mv_{t+1}Pv_hv_tPv_qv_1$ is a cycle of length at least $k$, a contradiction.
			If $v_l\in V(\widetilde{C})$, then $P'=v_tPv_{l+1}v_1Pv_lv_{h+1}Pv_mv_{t+1}Pv_h$ is an $(H,T)$-path with $d_{P'}(v_h)\ge \omega,$ a contradiction.
			Hence, $V(v_{t+1}Pv_m)= \{v_h,v_{h+1},v_m\}$.
			
			Next we show that $m=k$. Suppose $(s,t)$ is a minimal crossing pair. Then $N_P(v_m)=\{v_s,v_t,v_{h+1}\},$ i.e., $\delta=3.$ If $l\ne s+1$, then $v_1Pv_lv_{h+1}v_mv_tv_1$ is a cycle of length at least $k-1$, a contradiction. Hence, $l=s+1.$
			There is a cycle $v_1Pv_sv_mv_{h+1}v_{s+1}Pv_tv_1$. By $c(G)=k-1,$ $l=t-1.$ Thus, $|V(v_sPv_t)|=3$ and $m=k.$
			Suppose $(s,t)$ is not a minimal crossing pair. By \Cref{c17} and \Cref{c20}, $N_P(v_1)\cap V(v_sPv_t)=N_P(v_m)\cap V(v_sPv_t)=\{v_s,v_{s+2},...,v_{t-2},v_t\}.$ Thus, $m=k.$

			Let $A_1=V(v_1Pv_{s-1})$, 
			let $A_2=N_P(v_1)\cap V(v_sPv_t)\cup \{v_{h+1}\}$,
			let $A_2'=N_P(v_m)\cap V(v_sPv_t)\cup \{v_{h+1}\}$, let $A_3=V(v_{t+1}Pv_m)\setminus\{v_h,v_{h+1}\}$ and let $X=V(G)\setminus A_1\cup A_2\cup A_3.$
			
			By the previous proof, we have $A_3=\{v_m\}$ and $v_h\in X.$
			By (\ref{a4}), (\ref{a5}), \Cref{c17} and \Cref{c20}, $A_2=A_2'=\{v_s,v_{s+2},...,v_t\}\cup \{v_{h+1}\}$ with $t\equiv s\pmod 2$.
			
			We assert that each vertex of $X$ is only connected to $A_2\subseteq A_1\cup A_2\cup A_3$.
			Since $N_P[v_1]\subseteq A_1\cup A_2$ and $N_G[v_m]=A_2\cup A_3$, each vertex of $X\cap V(P)$ is only connected to $A_2\subseteq A_1\cup A_2\cup A_3$.
			Suppose $x\in X\setminus V(P)$ is connected to $A_1\cup A_2$.
			Since $G$ is 2-connected, there are two vertex-disjoint $(x,P)$-paths $P_1$ and $P_2$ with $V(P_i)\cap V(P)=\{x_i\}$ and $x \in X\setminus V(P)$ for $i=1,2$ Assume that $x_1\in V(v_1Px_2).$
			By the choice of $P$, $x_2\notin A_1.$ Suppose $x_1\in A_1$.
			By the choice of $P$, $[X\cap V(P),X\setminus V(P)]=\emptyset.$
			Then $x_2\in A_2$ and hence there exists an $m$-cycle, a contradiction.
			Thus, each vertex of $X$ is only connected to $A_2\subseteq A_1\cup A_2\cup A_3$.
			
			Similar to the proof of \Cref{c7'}, $X$ is an independent set.
			Since $\delta(G)\ge \delta$ and $v_1v_{h+1}\notin E(G)$, $G$ is a subgraph of $H_1(n,\omega+1,\delta)$, contradicting to the maximality of $G.$  Thus, $N_P[v_{h+1}]\setminus\{v_h\}\subseteq N_P[v_m]$ and $|N_P[v_{h+1}]\cap N_P[v_m]|=\delta-1$ since $d_P(v_{h+1})\ge \delta.$
			This proves \Cref{c20'}.
		\end{proof}
		
		Let $A_1=V(v_1Pv_{s-1})$, let $A_2=N_P(v_1)\cap V(v_sPv_t)\cup \{v_h\}$ and $A_2'=N_P(v_m)\cap V(v_sPv_t)\cup \{v_h\}$ if $h=t$ or $A_2=N_P(v_1)\cap V(v_sPv_t)$ and $A_2'=N_P(v_m)\cap V(v_sPv_t)$ if $h\neq t$, and let $A_3=V(v_{t+1}Pv_m)$ and let $X=V(G)\setminus (A_1\cup A_2\cup A_3)$. Clearly, $X\ne\emptyset.$
		
		By (\ref{a4}), (\ref{a5}), \Cref{c17} and \Cref{c20}, $A_2=A_2'=\{v_s,v_t\}$ when $(s,t)$ is a minimal crossing pair or $A_2=A_2'=\{v_s,v_{s+2},...,v_{t-2},v_t\}$ with $t\equiv s \pmod 2$ and $m=k$ when $(s,t)$ is not a minimal crossing pair.
		Clearly, $A_1$ is a clique.
		
		\begin{claim}\label{c21}
			Each vertex of $X$ is only connected to $A_2\subseteq A_1\cup A_2\cup A_3$; and if $|A_2|\geq 3$, then $X$ is an independent set.
		\end{claim}
		
		\begin{proof}
			We first prove that each vertex of $X\cap V(P)$ is only connected to $A_2\subseteq A_1\cup A_2\cup A_3$. By \Cref{c17}, for $1\le i\le s-1,$ $N_P[v_i]=N_P[v_1]\subseteq A_1\cup A_2$ and then $[A_1,X\cap V(P)]=\emptyset.$ By \Cref{c18}(i), for any vertex $x\in V(v_{t+1}Pv_m)\setminus \{v_h,v_{h+1}\}$, $N_P(x)\subseteq N_P[v_m]\cup \{v_h\}$. Since $N_P[v_m]\subseteq A_2\cup A_3$, $[ A_3\setminus \{v_h,v_{h+1}\}, X\cap V(P)]=\emptyset$. If $h\ne t,$ by \Cref{c18}(ii) and \Cref{c20'}, then $N_P[v_h]\cup N_P[v_{h+1}]\subseteq A_2\cup A_3$. If $h=t,$ by \Cref{c18}(iii), then $N_P[v_{h+1}]\subseteq A_2\cup A_3$. Hence, $[A_3,X\cap V(P)]=\emptyset.$ Thus, each vertex in $X\cap V(P)$ is only connected to $A_2\subseteq A_1\cup A_2\cup A_3$.
			
			Next we prove that each vertex of $X\setminus V(P)$ is only connected to $A_2\subseteq A_1\cup A_2\cup A_3$. For $t\le j\le m$ and $j\ne h,$ there is an $(H,T)$-path $v_1Pv_{j}v_mPv_{j+1}$.
			By the choice of $P$, we have $[X\setminus V(P), A_3\setminus \{v_{h+1}\}]=\emptyset$.
			Since $G$ is 2-connected, there are two vertex-disjoint $(x,P)$-paths $P_1$ and $P_2$ with $V(P_i)\cap V(P)=\{x_i\}$ and $x \in X\setminus V(P)$ for $i=1,2$. Assume that $x_1\in V(v_1Px_2).$ Clearly, $\{x_1,x_2\}\subseteq V(P)\setminus (A_3\setminus\{v_{h+1}\}).$

			Suppose $h\ne t.$  If $x_1,x_2\in A_1$, by \Cref{c17}, then there exists an $(H,T)$-path longer than $P$, a contradiction. Suppose $x_1\in A_1$ and $x_2\in A_2\cup (X\cap V(P))$. Without loss of generality, let $x_1=v_1$, $x_2=v_s$ and $|A_2|=2$. Other situations are similar. Then we find a $k$-cycle $v_1P_1xP_2v_sv_mPv_tv_{s-1}Pv_1$, a contradiction. Suppose $x_1\in A_1\cup (X\cap V(P))$ and $x_2=v_{h+1}$. Without loss of generality, let $x_1=v_1$. By \Cref{c18}(ii),  $v_sv_h\in E(G)$ and then there exists a cycle $v_1P_1xP_2v_{h+1}Pv_mv_tv_{t+1}Pv_hv_sPv_1$ of length at least $m$, a contradiction.
			Suppose $x_1\in A_2$ and $x_2=v_{h+1}.$ If $v_{h}\ne v_{t+1}$, then there is a cycle of length at least $k$, a contradiction. Thus, $v_{h}=v_{t+1}.$ Moreover, consider the cycle $v_1Pv_sv_mPv_{h+1}P_2xP_1x_1Pv_{s+2}x^{+2}_1Pv_tv_1$, then $P_1$ and $P_2$ are two edges. This implies that all neighbors of $x$ belongs to $A_2\cup \{v_{h+1}\}.$ Since $d_G(v_m)=\delta$ and $N_G(v_m)=A_2\cup (A_3\setminus\{v_h\}),$ we have $v_m=v_{h+2},$ that is, $V(v_{t+1}Pv_m)=\{v_h,v_{h+1},v_m\}.$ By the previous proof of \Cref{c20'}, a contradiction to the maximality of $G$.

			Suppose $h= t.$ By the choice of $P$, $x_1\notin A_1$ or $x_2\notin A_2.$
			If $x_1\in A_1$ and $x_2\in A_2\cup (X\cap V(P))$, then there is a cycle of length at least $ k,$ a contradiction. Suppose $x_1\in A_1\cup A_2\cup (X\cap V(P))$ and $x_2=v_{h+1}$.  Without loss of generality, let $x_1=v_1$. Other situations are similar. Then we find a cycle $v_1P_1xP_2v_{h+1}Pv_mv_sPv_tv_{s-1}Pv_1$ of length at least $m+1,$ a contradiction. Thus, each vertex of $X\setminus V(P)$ is only connected to $A_2\subseteq A_1\cup A_2\cup A_3$.
			
			Let $|A_2|\ge 3$. If $v_kv_f\in E(G)$ with $v_k\in X\setminus V(P) $ and $v_f\in X\cap V(P)$, then $v_1Pv_f^-v_mPv_fv_k$ is an $(H,T)$-path longer than $P$, a contradiction. Thus, $[X\setminus V(P),X\cap V(P)]=\emptyset.$
			The rest of the proof is similar to that of \Cref{c7'}. Thus, $X$ is an independent set of $G$.
			This proves \Cref{c21}.
		\end{proof}

		By \Cref{c17} and (\ref{a0}), $G[A_1\cup A_2]=K_\omega$ and $|A_2\cup A_3|=\delta+2$.
		Suppose that $(s,t)$ is a minimal crossing pair.
		Then $|A_1|=\omega-2$ and $|A_3|=\delta$.
		By \Cref{c21}, $G+v_hv_m$ contains no cycle of length at least $k$, contradicting to the maximality of $G.$
		Hence, $(s,t)$ is not a minimal crossing pair.
		Since $v_{h+1}\in A_3$, $|A_3|\geq 2$.
		By \Cref{c21}, $(|A_2|,|A_3|)=(\delta, 2)$, i.e., $h=t$.
		Now $G+v_hv_m$ contains no cycle of length at least $ k$, contradicting to the maximality of $G.$
		
		\vskip 3mm
		{\bf Case C.} $s+1\leq h<h+1\le t-1$.
		\vskip 3mm
		
		Clearly, $(s,t)$ is not a minimal crossing pair.
		By (\ref{a4}) and (\ref{a5}), $V(v_1Pv_{s-1})\subseteq N_P^-(v_1)$ and $V(v_{t+1}Pv_{m})\subseteq N_P^+(v_m)$. Then
		\begin{align}\label{a6}
			V(v_1Pv_{s})\subseteq N_P[v_1] ~\text{and} ~V(v_{t}Pv_{m})\subseteq N_P[v_m].
		\end{align}
		
		Recall that  $(p,q)$ is a minimum crossing pair and $\widetilde{C}=v_1Pv_pv_mPv_qv_1$ with $|\widetilde{C}|=k-1$. Since $v_h,v_{h+1}\in V(\widetilde{C})$, we have either $h+1\leq p<q\le t$ or $s\le p< q\leq h$.
		
		\begin{claim}\label{c22}
			$N_P[v_i]=N_P[v_1]$ for $1\le i\le s-1$. For $t+1\le j\le m$, either $N_P[v_j]=N_P[v_m]$ or $N_P[v_j]\setminus N_P[v_m]= \{v_h\}$ and $|N_P[v_m]\cap  N_P[v_j]|= \delta-1.$
		\end{claim}
		
		\begin{proof}
			Let $v_i\in V(v_1Pv_{s-1}).$
			By $N_P[v_1]=K_\omega$ and $(\ref{a6})$, $N_P[v_1]\subseteq N_P[v_i].$ Consider the path $P
			_1=v_iPv_1v_{i+1}Pv_m\in \mathcal P
			$. Then $d_P(v_i)=d_P(v_1)=\omega-1$, that is, $N_P[v_i]=N_P[v_1]$.
			
			Let $v_i\in V(v_{t+1}Pv_{m}).$
			If $N_P[v_j]=N_P[v_m]$, we are done, so assume that $N_P[v_j]\ne N_P[v_m]$. Suppose $y\in N_P[v_j]\setminus (N_P[v_m]\cup \{v_h\})$. If $y\in V(\widetilde{C}),$ by  $y\ne v_p$, $yv_m\notin E(G)$ and $(\ref{a4}),$ then $y^+v_1\in E(G)$. By $(\ref{a6})$,  $v_mv_{j-1}\in E(G)$ and then $v_1Pyv_jPv_mv_{j-1}Py^+v_1$ is an $m$-cycle, a contradiction. If $y\notin V(\widetilde{C}),$ then $v_1Pyv_jPv_mv_{j-1}Pv_qv_1$ is a  cycle of length at least $k$, a contradiction. Thus, $N_P[v_j]\setminus N_P[v_m]= \{v_h\}$. Consider the path $v_1Pv_{j-1}v_mPv_j\in \mathcal{P}.$ We have $d_P(v_j)=\delta.$ Clearly, $|N_P[v_m]\cap  N_P[v_j]|= \delta-1.$
			This proves \Cref{c22}.
		\end{proof}

		\begin{claim}\label{c23}
			$v_1$ and $v_m$ are not adjacent to any two consecutive vertices of $V(v_sPv_t).$
		\end{claim}
		
		\begin{proof}
			To the contrary, suppose $v_i,v_{i+1}\in N_P(v_1)\cap V(v_sPv_t)$.
			By \Cref{c22}, $v_iv_{s-1}\in E(G)$.
			Then $v_1Pv_{s-1}v_iPv_sv_mPv_{i+1}v_1$ is a cycle of length $m\ge k$, a contradiction.
			Suppose $v_j,v_{j+1}\in N_P(v_m)\cap V(v_sPv_t)$. By \Cref{c22}, $v_jv_{t+1}\in E(G)$ or $v_{j+1}v_{t+1}\in E(G)$.
			If $v_jv_{t+1}\in E(G)$, then $v_1Pv_jv_{t+1}Pv_mv_{j+1}Pv_tv_1$ is an $m$-cycle of length, a contradiction.
			If $v_{j+1}v_{t+1}\in E(G)$, then $v_1Pv_jv_mPv_{t+1}v_{j+1}Pv_tv_1$ is an $m$-cycle, a contradiction.
			This proves \Cref{c23}.
		\end{proof}

		Recall that $(p,q)$ is a minimum crossing pair.
		If $v_1v_h,v_{h+1}v_m\in E(G)$,  by \Cref{c23}, then $v_1v_{t-1}\notin E(G)$ and then there is a path $P'=v_hPv_1v_tPv_mv_{h+1}Pv_{t-1}\in \mathcal P$ with $d_{P'}(v_h)\ge \omega,$ a contradiction.
		Note that $v_1v_{h+1}\notin E(G)$ and $v_mv_h\notin E(G)$.
		Based on the possibilities for the edge set between $\{v_h,v_{h+1}\}$ and $\{v_1,v_m\}$, we just consider the following three subcases.

		\vskip 3mm
		\item[$\bullet$] ~$v_1v_h\notin E(G)$ and $v_{h+1}v_m\notin E(G)$.
		\vskip 3mm
		
		Let $A_1=V(v_1Pv_{s-1})$, let $A_2=N_P(v_1)\cap V(v_sPv_t)$, let $A_2'=N_P(v_m)\cap V(v_sPv_t)$, let $A_3=V(v_{t+1}Pv_m)$ and let $X=V(G)\setminus (A_1\cup A_2\cup A_3).$ 
		Clearly, $X\ne\emptyset.$
		By (\ref{a4}), (\ref{a5}), \Cref{c22} and \Cref{c23}, 
		\begin{align}\label{a7}
			A_2=A_2'=\{v_s,\dots,v_{h-3},v_{h-1},v_{h+2},v_{h+4},\dots,v_t\}.
		\end{align}
		Clearly, $t\equiv (h+2)\pmod 2$ and $h-1\equiv s\pmod 2$.
		Then $|V(v_pPv_q)|=3$, $m=k$ and $|A_2|\ge 3$.
		
		\begin{claim}\label{c24}
			$[X\cap V(P),X\setminus V(P)]=\emptyset.$
		\end{claim}
		\begin{proof}
			Consider two paths  $v_1Pv_{h-1}v_mPv_h\in \mathcal{P}$ and $v_1Pv_{s-1}v_{h+2}Pv_mv_sPv_{h+1}\in \mathcal{P}$.
			By the choice of $P$, $[\{v_h,v_{h+1}\},X\setminus V(P)]=\emptyset.$
			For any vertex $v_f\in (X\cap V(P))\setminus\{v_h,v_{h+1}\},$ we have $v_{f-1}v_m\in E(G).$ Consider the path $v_1Pv_{f-1}v_mPv_f$. Then $[\{v_f\},X\setminus V(P)]= \emptyset$.
			By the arbitrary of $v_f$, $[(X\cap V(P))\setminus\{v_h,v_{h+1}\},X\setminus V(P)]=\emptyset.$ Thus, $[X\cap V(P),X\setminus V(P)]=\emptyset.$
			This proves \Cref{c24}.
		\end{proof}

		\begin{claim}\label{c25}
			$G[X\cap V(P)]$ contains exactly one edge $v_hv_{h+1}$.
		\end{claim}

		\begin{proof}
			We first show that $[\{v_h,v_{h+1}\},(X\cap V(P))\setminus\{v_h,v_{h+1}\}]=\emptyset.$
			By $|A_2|\ge 3$, $ (X\cap V(P))\setminus\{v_h,v_{h+1}\}\ne \emptyset.$
			To the contrary, suppose $x\in (X\cap V(P))\setminus\{v_h,v_{h+1}\}$. By (\ref{a7}), we have $x^+,x^-\in N_P(v_1)\cap N_P(v_m)$.
			If $xv_h\in E(G)$, by (\ref{a7}), then $v_1Pxv_hPv_mx^+Pv_{h-1}v_1$ when $x\in V(v_1Pv_h)$ or $v_1Pv_{h-1}v_mPx^+x\allowbreak v_hv_{h+1}Px^-v_1$ when $x\in V(v_{h+1}Pv_m)$ is an $m$-cycle.
			If $xv_{h+1}\in E(G)$, then $v_1Pxv_{h+1}Px^+ v_mPv_{h+2}v_1$ when $x\in V(v_1Pv_h)$ or $v_1Pv_{h+1}x x^-Pv_{h+2} v_mPx^+v_1$ when $\allowbreak x\in V(v_{h+1}Pv_m)$ is an $m$-cycle. Both situations lead to a contradiction.
			Thus, $[\{v_h,v_{h+1}\},(X\cap V(P))\setminus\{v_h,v_{h+1}\}]=\emptyset$.
			
			Now we show that $(X\cap V(P))\setminus\{v_h,v_{h+1}\}$ is an independent set.
			If  $|X\cap V(P)|=3$, we are done, so assume that $|X\cap V(P)|\ge 4.$
			Let $x, y\in X\cap V(P)\setminus \{v_h,v_{h+1}\}$ and $xy \in E(G)$.
			By (\ref{a7}), $x^+v_1, y^-v_m\in E(G)$
			and then $v_1PxyPv_my^-P x^+v_1$ is a cycle of length $m\ge k$, a contradiction.
			This proves \Cref{c25}.
		\end{proof}

		\begin{claim}\label{c26}
			Each vertex of $X$ is only connected to $A_2\subseteq A_1\cup A_2\cup A_3$ and  $A_3=\{v_m\}$.
		\end{claim}
		
		\begin{proof}
			We first prove that each vertex of $X\cap V(P)$ is only connected to $A_2\subseteq A_1\cup A_2\cup A_3$. By \Cref{c22} and $N_P[v_1]\subseteq A_1\cup A_2,$ $[A_1,X\cap V(P)]=\emptyset.$ Clearly, $(X\cap V(P))\setminus\{v_h,v_{h+1}\}\ne\emptyset$. Suppose $x\in (X\cap V(P))\setminus\{v_h,v_{h+1}\}$. By \Cref{c22}, $[\{x\},A_3]=\emptyset$, and then $[A_3,(X\cap V(P))\setminus\{v_h,v_{h+1}\}]=\emptyset.$ By the choice of $P$ and \Cref{c24}, $N_G(x)=N_P(x)\subseteq A_2$. Since $d_P(x)\ge \delta$ and $|A_2\cup A_3|=\delta+1$, we have $|A_2|= \delta$ and $|A_3|=1$. Hence, $A_3=\{v_m\}$. By $v_h,v_{h+1}\notin N_P(v_m)$, $[X\cap V(P),A_1\cup A_3]
			=\emptyset$. Thus, each vertex of $X\cap V(P)$ is only connected to $A_2\subseteq A_1\cup A_2\cup A_3$.
			
			Next we prove that each vertex of $X\setminus V(P)$ is only connected to $A_2\subseteq A_1\cup A_2\cup A_3$. Clearly, $[X\setminus V(P),A_3]=\emptyset$.
			Suppose that $x \in X\setminus V(P)$ is connected to $A_1\cup A_2\subseteq A_1\cup A_2\cup A_3$. Since $G$ is 2-connected, there are two vertex-disjoint $(x,P)$-paths $P_1$ and $P_2$ with $V(P_i)\cap V(P)=\{x_i\}$  for $i=1,2$.  Let $x_1\in V(v_1Px_2)$.
			Suppose $x_1\in A_1$. Without loss of generality, let $x_1=v_1$. Other cases are similar. By the choice of $P$ and $N_P[v_1]=K_\omega$, $x_2\notin A_1\cup \{v_s\}$.
			If $x_2\in A_2\setminus \{v_s,v_{h+2}\}$, then $v_1P_1xP_2x_2Pv_mx_2^{-2}Pv_1$ is a cycle of length at least $k$. If $x_2=v_{h+2}$, then $v_1P_1xP_2v_{h+2}Pv_sv_mPx_2^{+2}v_{s-1}Pv_1$ is a cycle of length at least $k$ when $t\ne h+2$, or $v_{s-1}Pv_1P_1xP_2v_tPv_sv_mPv_{t+1}$ is an $(H,T)$-path longer than $P$ when $t=h+2$, a contradiction. Hence, $x_1\in A_2$, similarly, $x_2\in A_2$.  Thus, each vertex of $X\setminus V(P)$ is only connected to $A_2\subseteq A_1\cup A_2\cup A_3$.
			This proves \Cref{c26}.
		\end{proof}
		
		By \Cref{c24}, \Cref{c25} and \Cref{c26}, $A_3=\{v_m\}$ and it is easy to verify that $X\setminus V(P)$ consists of some $K_2$ or some $K_1.$ If $G[X]$ contains exactly one edge $v_hv_{h+1}$, since
		$\omega(G)=\omega$ and $\delta(G)=\delta$, we have $\omega>\delta.$ Then $G=G_1$.
		Assume that $G[X]$ contains another edge $y_1y_2\ne v_hv_{h+1}$.
		By the maximality of $G$ and $c(G)=k-1$, $N_G[y_1]=N_G[y_2]=\{y_1,y_2,v_{h-1},v_{h+2}\}.$ Now $\delta(G)=3$ and $\omega(G)=\omega\ge 4$. Then
		$G=H_4(n,\omega,3)$, where $l_1\ge 1$ and $l_2\ge 2.$

		\vskip 3mm
		\item[$\bullet$]~$v_1v_h\notin E(G)$ and $v_{h+1}v_m\in E(G)$.
		\vskip 3mm
		
		Let $A_1=V(v_1Pv_{s-1})$, let $A_2=N_P(v_1)\cap V(v_sPv_t)\cup\{v_{h+1}\}$, let $A_2'=N_P(v_m)\cap V(v_sPv_t)\cup\{v_{h+1}\}$, let $A_3=V(v_{t+1}Pv_m)$ and let $X=V(G)\setminus (A_1\cup A_2\cup A_3).$
		By (\ref{a4}), (\ref{a5}), \Cref{c22} and \Cref{c23}, $v_{h-1}\in A_2$ and $v_{h+2}\in X\cap V(P)$.
		Similarly, we have
		\begin{align}
			A_2=A_2'=\{v_s,\dots,v_{h-3}, v_{h-1}, v_{h+1},v_{h+3},\dots,v_t\}.\nonumber
		\end{align}
		Clearly, $|A_2|\ge 3$ and $m=k.$ Note that $A_2\setminus\{v_{h+1}\}\subseteq N_P(v_1)$ and $A_2\subseteq N_P(v_m)$.

		\begin{claim}\label{c27}
			$[X\cap V(P), X\setminus V(P)]=\emptyset$ and each vertex of $X$ is only connected to $A_2\subseteq A_1\cup A_2\cup A_3$. Moreover, $X$ is an independent set.
		\end{claim}
		
		\begin{proof}
			We first show that $[X\cap V(P), X\setminus V(P)]=\emptyset$. For any vertex $v_f\in X\cap V(P),$ we have $v_{f-1}\in N_P(v_m)$. Consider the path $v_1Pv_{f-1}v_mPv_f\in \mathcal{P}.$ Then $[\{v_f\},X\setminus V(P)]=\emptyset,$ that is, $[X\cap V(P), X\setminus V(P)]=\emptyset$.
			
			Next, we show that each vertex of $X$ is only connected to $A_2\subseteq A_1\cup A_2\cup A_3$.
			Consider the path $v_1Pv_{j-1}v_mPv_j\in \mathcal{P}$ for $t+1\le j\le m.$ Then $[X\setminus V(P),A_3]=\emptyset.$ Suppose that $x\in X\setminus V(P)$ is connected to $A_1\cup A_2\subseteq A_1\cup A_2\cup A_3.$ Since $G$ is 2-connected, there are two vertex-disjoint $(x,P)$-paths $P_1$ and $P_2$ with $V(P_i)\cap V(P)=\{x_i\}$ and $x \in X\setminus V(P)$ for $i=1,2$. Let $x_1\in V(v_1Px_2)$. Suppose $x_1\in A_1$. By the choice of $P$, $x_2\notin A_1.$ If $x_2\in A_2,$ then there is a cycle of length at least $k$, a contradiction. Hence, $x_1\in A_2$, similarly, $x_2\in A_2$.
			Thus, each vertex of $X\setminus V(P)$ is only connected to $A_2\subseteq A_1\cup A_2\cup A_3.$
			
			Now, we prove that each vertex of $X\cap V(P)$ is only connected to $A_2\subseteq A_1\cup A_2\cup A_3.$
			Since $N_G(v_m)=N_P(v_m)\subseteq A_2$, we have $[X\cap V(P),A_3]=\emptyset$. By \Cref{c22} and $N_P[v_1]\subseteq A_1\cup A_2,$  $[X\cap V(P),A_1]=\emptyset$, that is, each vertex of $X\cap V(P)$ is only connected to $A_2\subseteq A_1\cup A_2\cup A_3.$ Thus, each vertex of $X$ is only connected to $A_2\subseteq A_1\cup A_2\cup A_3$.
			
			It is easy to verify that for any pair vertices $y_1,y_2\in X\cap V(P)$, there is a $(y_1,y_2)$-path of order $m.$ Then $X\cap V(P)$ is an independent set.
			For any two vertices $y_1,y_2\in A_2$, there exits a $(y_1,y_2)$-path of order $k-2$. Note that $[X\cap V(P), X\setminus V(P)]=\emptyset$ and each vertex of $X\setminus V(P)$ is only connected to $A_2\subseteq A_1\cup A_2\cup A_3.$ Since $G$ is 2-connected, $X\setminus V(P)$ is an independent set.
			By $[X\cap V(P), X\setminus V(P)]=\emptyset$, $X$ is an independent set. This proves \Cref{c27}.
		\end{proof}
		
		By $\delta(G)= \delta$, \Cref{c27} and $|A_2\cup A_3|=\delta+1$, we have  $|A_1|=\omega-\delta+1$, $|A_2|=\delta$, $|X|=n-\omega-2$ and $|A_3|=1$.
		Now $G$ is obtained from $H_1(n,\omega+1,\delta)$ by deleting all edges between $K_{\omega-\delta+1}$ and a common vertex of $K_\delta$, contradicting to the maximality of $G$.
		
		\vskip 3mm
		\item[$\bullet$]~$v_1v_h\in E(G)$ and $v_{h+1}v_m\notin E(G)$.
		\vskip 3mm
		
		Let $A_1=V(v_1Pv_{s-1})$, let $A_2=N_P(v_1)\cap V(v_sPv_t)\cup\{v_h\}$,
		let $A_2'=N_P(v_m)\cap V(v_sPv_t)\cup\{v_h\}$, let $A_3=V(v_{t+1}Pv_m)$ and let $X=V(G)\setminus (A_1\cup A_2\cup A_3).$ Clearly, $X\ne\emptyset.$
		By \Cref{c23} and $c(G)=k-1$, $v_{h-1}\notin N_P(v_1)\cup N_P(v_m).$
		By (\ref{a4}),(\ref{a5}) and \Cref{c23}, $v_{h-2},v_{h+2}\in A_2$.
		Similarly, we have
		\begin{align}
			A_2=A_2'=\{v_s,\dots,v_{h-2},v_h,v_{h+2},\dots,v_t\}.\nonumber
		\end{align}
		Clearly, $|A_2|\ge 3$ and $m=k.$ Note that $A_2\subseteq N_P(v_1)$ and $A_2\setminus\{v_{h}\}\subseteq N_P(v_m)$.
		
		\begin{claim}\label{c28}
			$[X\cap V(P), X\setminus V(P)]=\emptyset$ and each vertex of $X$ is only connected to $A_2\subseteq A_1\cup A_2\cup A_3$. Moreover, $X$ is an independent set.
		\end{claim}
		
		\begin{proof}
			We first show that $[X\cap V(P), X\setminus V(P)]=\emptyset$. For any vertex $v_f\in (X\cap V(P))\setminus\{v_{h+1}\},$ we have $v_{f-1}\in A_2\setminus\{v_h\}\subseteq N_P(v_m).$  Consider the path $v_1Pv_{f-1}v_mPv_f\in \mathcal{P}.$ By the choice of $P$, $[\{v_f\},X\setminus V(P)]=\emptyset,$ that is, $[(X\cap V(P))\setminus \{v_{h+1}\}, X\setminus V(P)]=\emptyset$.
			Next, we show that $[\{v_{h+1}\},X\setminus V(P)]=\emptyset$.
			
			Suppose there is a vertex $v_g\in N(v_{h+1})\cap (X\setminus V(P)).$  We show that $[\{v_g\},A_1\cup A_3\cup X]=\emptyset$.
			By (\ref{a6}) and the choice of $P$, $[\{v_g\},A_3]=\emptyset$.
			Suppose $[\{v_g\}, A_1]\neq \emptyset.$
			Without loss of generality, let $v_1\in N(v_g)\cap A_1$.
			Then $v_1v_gv_{h+1}\allowbreak Pv_mv_{h}Pv_1$ 
			is a $(k+1)$-cycle, a contradiction. Other cases are similar.
			Thus, $[\{v_g\},A_1]=\emptyset$. Consider the path $v_1Pv_{h-2}v_mPv_{h+2}v_hv_{h+1}v_g\in \mathcal{P}.$ Then $[\{v_{g}\}, X\setminus V(P)]=\emptyset$.
			We assert that $[\{v_g\},X\cap V(P)]=\emptyset$. Suppose $v_lv_g\in E(G)$ with $v_l \in X\cap V(P)$.
			If $s+1\le l\le h-1,$ then $v_hPv_{l+1}v_1Pv_{l}v_gv_{h+1}Pv_m$ is a longer $(H,T)$-path  than $P$, a contradiction. If $h+2=t$, we are done, so assume that $h+2<t$. By the definition of $A_2$, we have $h+3\le t-1$. If $h+3\le l\le t-1,$ then $v_1Pv_{h+1}v_gv_lPv_mv_{l-1}Pv_{h+2}v_1$ is a $(k+1)$-cycle, a contradiction.
			Hence, $[\{v_g\},A_1\cup A_3\cup X]=\emptyset$.
			
			By $d_G(v_g)\ge \delta\ge 2$, there is a vertex $v_r\in N(v_g)\cap A_2$. If $r\notin \{h-2,h,h+2\}$, then there exists a $k$-cycle. If $r\in \{h,h+2\}$, then  there exists an $(H,T)$-path longer than $P.$ If $r=h-2,$ then there is a path $P'=v_hv_1Pv_rv_gv_{h+1}Pv_m\in \mathcal{P}$ with $d_{P'}(v_h)\ge \omega.$ Both situations lead to a contradiction. Thus, $[\{v_{h+1}\},X\setminus V(P)]=\emptyset$ and then $[X\cap V(P), X\setminus V(P)]=\emptyset$.
			
			We simply prove that each vertex of $X$ is only connected to $A_2\subseteq A_1\cup A_2\cup A_3$ and $X$ is an independent set. By \Cref{c22}, $N_G[A_3]\subseteq A_2\cup A_3$ and $N_P[A_1]\subseteq A_1\cup A_2$. Then $[X,A_3]=\emptyset$ and $[X\cap V(P),A_1]=\emptyset$. Since $G$ is 2-connected, there are two vertex-disjoint $(x,P)$-paths $P_1$ and $P_2$ with $V(P_i)\cap V(P)=\{x_i\}$ and $x \in X\setminus V(P)$ for $i=1,2$. Assume that $x_1\in V(v_1Px_2)$. If $x_1,x_2\in A_1$, there is an $(H,T)$-path longer than $P.$ If $x_1\in A_1$ and $x_2\in A_2\setminus\{v_{h+2}\}$, there is an $m$-cycle. If $x_1\in A_1$ and $x_2=v_{h+2}$, there is an $(H,T)$-path $P'=v_hPx_1^+v_1Px_1P_1xP_2v_{h+2}Pv_m$ with $d_{P'}(v_h)\ge \omega.$ Both situations lead to a contradiction. Thus, each vertex of $X$ is only connected to $A_2\subseteq A_1\cup A_2\cup A_3$. Similar to the proof of \Cref{c27}, $X$ is an independent set. This proves \Cref{c28}.
		\end{proof}

		Since $G[A_1\cup A_2]=K_\omega$, $|A_2\cup A_3|-1=|N_P[v_m]|=\delta+1$ and $\delta(G)= \delta$, we have $(|A_1|,|A_2|,|A_3|)= (\omega-\delta-1,\delta+1,1)$ or $(\omega-\delta,\delta,2)$.
		
		If $(|A_1|,|A_2|,|A_3|)=(\omega-\delta-1,\delta+1,1)$, then $G$ is obtained from $H_1(n,\omega,\delta+1)$ by deleting one edge between $K_{\delta+1}$ and $\overline{K}_{n-\omega}.$ This contradicts the assumption that $G$ is an edge-maximal graph since $c(H_1(n,\omega,\delta+1))=k-1$.
		
		If $(|A_1|,|A_2|,|A_3|)=(\omega-\delta,\delta,2)$, then $G$ is obtained from $G_1$ by deleting one edge between $K_{\delta}$ and $K_2.$ This contradicts the assumption that $G$ is an edge-maximal graph since $c(G_1)=k-1$.
		\vskip 3mm
		
		Hence, we find all edge-maximal graphs with circumference $\omega+\delta$ in the case of $d_P(v_1)=\omega-1$ and $d_P(v_m)=\delta$. By \Cref{l2.3}, the edge-maximal graph $G\in \{H_3(n,\omega,2),H_4(n,\omega,3)\}\cup\mathcal F.$
		Suppose $F$ is a graph satisfying the conditions of {\blue Theorem \ref{thm1.3}}. According to the previous proof, it follows that $c(F)\ge \min\{n,\omega+\delta+1\}$,  unless $F\in\{H(n,\omega,\delta),Z(n,\omega,\delta)\}$ when $c(F)=\omega+\delta-1$ or $F\in \{H_3(n,\omega,2),H_4(n,\omega,3)\}\cup \mathcal{G}$ when $c(F)=\omega+\delta$.

		This completes the proof of {\blue Theorem \ref{thm1.3}}.
	\end{proof}
	
	\section*{\normalsize Acknowledgement}  The authors are grateful to Professor Xingzhi Zhan for his constant support and guidance. This research  was supported by the NSFC grant 12271170.
	
	\section*{\normalsize Declaration}
	
	\noindent\textbf{Conflict~of~interest}
	The authors declare that they have no known competing financial interests or personal relationships that could have appeared to influence the work reported in this paper.
	
	\noindent\textbf{Data~availability}
	No data was used for the research described in the article.

\end{document}